\documentclass{amsart}

\usepackage{amsmath}
\usepackage{amssymb}
\usepackage{mathrsfs}
\usepackage{url}

\long\def\symbolfootnote[#1]#2{\begingroup%
\def\thefootnote{\fnsymbol{footnote}}\footnote[#1]{#2}\endgroup}

\newtheorem{theorem}{Theorem}[section]
\newtheorem{lemma}[theorem]{Lemma}
\newtheorem{proposition}[theorem]{Proposition}
\newtheorem{corollary}[theorem]{Corollary}

\theoremstyle{definition}
\newtheorem{definition}[theorem]{Definition}

\theoremstyle{remark}
\newtheorem{remark}[theorem]{Remark}

\numberwithin{equation}{section}
\def\pp{{\mathfrak{p}}}
\def\qq{{\mathfrak{q}}}
\def\SRes{{\mathrm{SRes}}}
\def\ord{{\mathrm{ord}}}
\def\mult{{\mathrm{mult}}}
\def\Cc{\mathcal{C}}\def\Dc{\mathcal{D}}
\def\fb{\mathbf{f}} 
\def\mm{\mathfrak{m}}
\def\ZZ{\mathbb{Z}} \def\NN{\mathbb{N}}
\def\Hom{\mathrm{Hom}}

\def\morl{\mathrm{morl}}
\def\Det{\mathrm{Det}}

\def\sylv{\mathrm{sylv}}

\def\PP{\mathbb{P}}
\def\KK{\mathbb{K}}
\def\Res{\mathrm{Res}}
\newcommand{\ud}[1]{\underline{#1}}

\def\Det{{\mathrm{Det}}}
\def\Lc{{\mathcal{L}}}
\def\Pc{{\mathcal{P}}}

\title[On the equations of the moving curve ideal]{On the equations of the moving curve ideal of a rational algebraic plane curve}
\author{Laurent Bus\'e}
\address{Galaad, INRIA,
2004 route des Lucioles, B.P.~93,
06902 Sophia Antipolis Cedex, France. Email: {\tt Laurent.Buse@inria.fr}}

\begin{document}

\maketitle

\begin{abstract}
Given a parametrization of a rational plane algebraic curve $\Cc$, some explicit adjoint pencils on $\Cc$ are described in terms of determinants. Moreover, some generators of  the Rees algebra  associated to this parametrization are presented. The main ingredient developed in this paper is a detailed  study of the elimination ideal of two homogeneous polynomials in two homogeneous variables that form a regular sequence. 
\end{abstract}

\section{Introduction}
Suppose we are given a rational map
\begin{eqnarray}\label{phi}
 \PP^1 & \xrightarrow{\phi} & \PP^2 \\ \nonumber
(X_1:X_2) & \mapsto & (g_1:g_2:g_3)(X_1,X_2)
\end{eqnarray}
where $g_1,g_2,g_3$ are homogeneous polynomials of degree $d\geq 1$ in the polynomial ring $\KK[X_1,X_2]$ with $\KK$ a field.
We assume that $g_1,g_2,g_3$ are not all zero and that the greatest common divisor of $g_1,g_2,g_3$ over $\KK[X_1,X_2]$ has degree $< d$, so that the closed image of the rational map $\phi$ is a rational algebraic plane curve $\Cc$. 

The geometric modeling community is interested in the manipulation of para\-me\-tri\-zed algebraic plane curves and has developed many tools for this purpose in the last decade. One of them
is what is called the \emph{moving curve ideal} \cite{Cox07}. Denoting by $\ud{T}=(T_1,T_2,T_3)$ the homogeneous coordinates of $\PP^2_\KK$, a \emph{moving curve} of degree $\nu\geq 0$ is  a polynomial
$$ \sum_{\substack{\alpha_1,\alpha_2,\alpha_3\geq 0 \\
\alpha_1+\alpha_2+\alpha_3=\nu}} A_{\alpha_1,\alpha_2,\alpha_3}(X_1,X_2)\, T_1^{\alpha_{1}}T_2^{\alpha_{2}}T_3^{\alpha_{3}}$$
where $A_{\alpha_1,\alpha_2,\alpha_3}(X_1,X_2) \in \KK[X_1,X_2]$. Such a moving curve is said to \emph{follow the parametrization $\phi$} if
$$\sum_{\substack{\alpha_1,\alpha_2,\alpha_3\geq 0 \\
\alpha_1+\alpha_2+\alpha_3=\nu}} A_{\alpha_1,\alpha_2,\alpha_3}(X_1,X_2)\, g_1(X_1,X_2)^{\alpha_{1}}g_2(X_1,X_2)^{\alpha_{2}}g_3(X_1,X_2)^{\alpha_{3}}=0.$$
The set of all moving curves that follow the parametrization $\phi$ form an ideal in the polynomial ring $\KK[X_1,X_2][T_1,T_2,T_3]$. It is called the \emph{moving curve ideal} of the parametrization $\phi$.

From an algebraic point of view, the moving curve ideal of $\phi$ can be seen as the defining ideal of the Rees Algebra of the ideal $I=(g_1,g_2,g_3)$ in $\KK[X_1,X_2]$. More precisely,  $\mathrm{Rees}_{\KK[X_1,X_2]}(I)$ is  the image of the $\KK[X_1,X_2]$-algebra morphism
$$ \KK[X_1,X_2][T_1,T_2,T_3] \xrightarrow{\beta} \KK[X_1,X_2][Z] : T_i \mapsto g_iZ$$
and the kernel of $\beta$ is exactly the moving curve ideal of $\phi$ 
(see for instance \cite[Proposition 3.5]{BuJo03} for a detailed proof of this well-known fact).
Notice that this ideal is naturally bi-graded: it is $\NN$-graded with respect to the homogeneous  
variables $T_1,T_2,T_3$ by definition, and it is also $\NN$-graded with respect to the variables $X_1,X_2$ because the polynomials $g_1,g_2,g_3 \in \KK[X_1,X_2]$ are homogeneous.
 
\medskip
\subsection*{Content of the paper}
In this paper, the moving curve ideal is studied in order to address two problems that have been recently raised by several authors. Our approach is based on the theory of \emph{inertia forms} for which we will have to develop new results.

First, we will focus on the determination of the 
equations of the moving curve ideal, that is to say on the computation of a system of generators as an ideal in $\KK[X_1,X_2][T_1,T_2,T_3]$. As we have already noticed, this corresponds to the determination of the equations of a certain Rees algebra and there is a vast literature on this topic -- see for instance \cite{Vas94} and the references therein. In our more precise context, this question of getting a full system of generators for the moving curve ideal of $\phi$ appears in \cite{Cox07,HSV,HW} where answers are given for a particular class of curves. In Section \ref{eqRees} of this paper, we will recover these results and obtain a full system of generators of the moving curve ideal for a new class of curves. More generally, we will provide new results on the character of some of the generators of the moving curve ideal of any rational curve.

Then, we will focus on the study of a certain graded part of the moving curve ideal, namely the moving curves following $\phi$ that are linear form in the variables $X_1,X_2$. Indeed,  David Cox recently observed in \cite{Cox07} that this graded part carries a lot of geometric properties of the curve $\Cc$. More precisely, \cite[Conjecture 3.8]{Cox07} suggests a close relation between adjoint pencils on $\Cc$ and moving curves following $\phi$ of degree 1 in $X_1,X_2$ and degree $d-2$ (resp.~$d-1$) in $T_1,T_2,T_3$. In Section \ref{adj}, we will investigate this relation and prove several new results. The main contribution is to show that under suitable genericity conditions any moving curves following $\phi$ of degree 1 in $X_1,X_2$ and degree $d-2$ or $d-1$ in $T_1,T_2,T_3$ is an adjoint pencil on $\Cc$. In general, we will show that
one can always find an adjoint pencil on $\Cc$ of degree $d-2$ in $T_1,T_2,T_3$ that belongs to the moving curve ideal of $\phi$, this adjoint pencil being described very simply in terms of certain determinants. 
Finally, as a by product of our study, we will obtain an extension of Abhyankar's \emph{Taylor resultant} \cite[Lecture 19 ,Theorem p.153]{Ab} from the polynomial parametrization case to the rational parametrization case.

\subsection*{Another description of the moving curve ideal}
Our approach to study the moving curve ideal is based on the following alternative description of this ideal.
The first syzygy module of $g_1,g_2,g_3$ is known to be a free homogeneous $\KK[X_1,X_2]$-module of rank 2 and a basis of this syzygy module consists in two homogeneous elements (the notation $\ud{X}$ and $\ud{T}$ stand for the set of variables $X_1,X_2$ and $T_1,T_2,T_3$ respectively)
$$p=(p_1(\ud{X}),p_2(\ud{X}),p_3(\ud{X})), \ q=(q_1(\ud{X}),q_2(\ud{X}),q_3(\ud{X}))  \in \KK[X_1,X_2]^3$$
of degree $\mu$ and $d-\mu$ respectively. This is a consequence of the Hilbert-Burch Theorem. Notice that the choice of $p$ and $q$ is not unique, but their degrees are fixed and  depend only on the parametrization $\phi$. We will identify $p,q$ with the two polynomials 
\begin{align}\label{pq}
p(\ud{X},\ud{T}) &= p_1(\ud{X})T_1+p_2(\ud{X})T_2+p_3(\ud{X})T_3 \\ \nonumber
q(\ud{X},\ud{T}) &= q_1(\ud{X})T_1+q_2(\ud{X})T_2+q_3(\ud{X})T_3
\end{align}
in $\KK[\ud{X},\ud{T}]$ that form what has been called a \emph{$\mu$-basis of the parametrization $\phi$} by the geometric modeling community because of its importance to handle parametrized plane curves.  

As an illustration of the benefit of $\mu$-bases, we recall the following well-known formula that relates the resultant of a $\mu$-basis, an implicit equation $C(T_1,T_2,T_3)$ of $\Cc$ and the degree of the parametrization \eqref{phi}:
$$\Res_{X_1:X_2}(p,q)=\alpha C(T_1,T_2,T_3)^{\deg(\phi)} \in \KK[\ud{T}]$$
where $\alpha \in \KK \setminus \{ 0\}$. 
Here is another characterization of the moving curve ideal of $\phi$.
\begin{proposition}\label{pqTF} Let $(p,q)$ by a $\mu$-basis of the parametrization $\phi$. Then, the sequence $(p,q)$ is regular in the ring $\KK[\ud{X},\ud{T}]$ and the moving curve ideal of $\phi$ is equal to the 
elimination ideal $\left( (p,q):(X_1,X_2)^\infty \right)$ in $\KK[\ud{X},\ud{T}]$.
\end{proposition}
\begin{proof}
The first assertion follows from \cite[\S 2.1]{ASV81}, and the second from \cite[Proposition 3.6]{BuJo03} for instance.
\end{proof}

This result shows that the study of the moving curve ideal is equivalent to the study of the inertia forms of two homogeneous polynomials in two homogeneous variables that form a regular sequence. Therefore, in Section \ref{inform} we provide a detailed study of these inertia forms. 

\section{Inertia forms of two polynomials in two homogeneous variables}\label{inform}

In this section we suppose given a non-zero commutative ring $A$ and two homogeneous polynomials
$$f_1(X_1,X_2)=U_0X_1^{d_1}+\cdots+U_{d_1}X_2^{d_1}, \ \ f_2(X_1,X_2)=V_0X_1^{d_2}+\cdots+V_{d_2}X_2^{d_2}$$
in the (canonically graded) polynomial ring $C=A[X_1,X_2]$ with respective degree $d_1, d_2$ such that $1\leq d_1 \leq d_2$. We will denote by $\mm$, resp.~$I$, the ideal of $C$ generated by $X_1,X_2$, resp.~$f_1,f_2$, and by $B$ the (canonically graded) quotient ring $C/I$. Also, we define the  integer $\delta=d_1+d_2-2\geq 0$. 

We recall that the resultant of $f_1$ and $f_2$, denoted $\Res(f_1,f_2)$, is equal to the determinant of the well-known Sylvester matrix 
$$
\underbrace{\left(\begin{array}{ccc}
U_0 &  & 0 \\
 & \ddots &  \\
\vdots & &  U_0 \\
  & & \\ 
U_{d_1} & & \vdots \\
 & \ddots & \\
0 & & U_{d_1}
\end{array}\right.
}_{d_2}
\underbrace{\left.\begin{array}{ccc}
V_0 &  & 0 \\
 & \ddots &  \\
\vdots & &  V_0 \\
  & & \\ 
V_{d_2} & & \vdots \\
 & \ddots & \\
0 & & V_{d_2}
\end{array}\right)
}_{d_1}$$
which is of size $d_1+d_2=\delta+2$. The \emph{first-order subresultants} of $f_1$ and $f_2$ correspond to some $\delta$-minors of this Sylvester matrix. More precisely, by expanding the determinant 
$$
\underbrace{\left|\begin{array}{ccc}
U_0 &  & 0 \\
 & \ddots &  \\
\vdots & &  U_0 \\
  & & \\ 
U_{d_1} & & \vdots \\
 & \ddots & \\
0 & & U_{d_1}
\end{array}\right.
}_{d_2-1}
\underbrace{\left.\begin{array}{ccc}
V_0 &  & 0 \\
 & \ddots &  \\
\vdots & &  V_0 \\
  & & \\ 
V_{d_2} & & \vdots \\
 & \ddots & \\
0 & & V_{d_2}
\end{array}\right.
}_{d_1-1}
\left.\begin{array}{c}
 T_0 \\
 \\
\vdots \\
 \\
\\
\vdots \\
\\
T_{\delta}
\end{array}\right|$$
along its last column, we get the polynomial 
$$\sum_{i=0}^{\delta}  \SRes_{\delta-i}(f_1,f_2)\, T_i \in A[T_0,\ldots,T_{\delta}].$$
The elements $\SRes_i(f_1,f_2) \in A$, $i=0,\ldots,\delta$, are the (first-order) subresultants of $f_1$ and $f_2$. The element $\SRes_0(f_1,f_2)$ is usually called the principal subresultant. 

\medskip

From now on in this section, we \emph{assume that $(f_1,f_2)$ is a $C$-regular sequence}. Our aim is to give, under suitable other conditions, an explicit description of the ideal of \emph{inertia forms} of $f_1,f_2$ w.r.t.~the variables $X_1,X_2$, that is an explicit description of the ideal $(I:_C \mm^\infty)$ consisting of all the elements $f\in C$ such that there exists an integer $k$ with the property that $\mm^kf\subset (f_1,f_2)$. Since $f_1$ and $f_2$ are obvious inertia forms, it is sufficient to describe the graded $B$-module $H^0_\mm(B) \subset B$ since we have the canonical isomorphism
 $$(I:\mm^\infty)/I \xrightarrow{\sim} H^0_\mm(B)=\{f \in B \text{ such that } \exists k \in \NN : \mm^kf=0 \}.$$
The methods we will use are inspired by \cite{Jou97} where such a work has been done for two generic homogeneous polynomials of the same degree, that is in the case $A=\ZZ[U_0,\ldots,U_{d_1},V_0,\ldots,V_{d_2}]$ with $d_1=d_2$.

\subsection{Sylvester forms, Morley forms and an explicit duality}

 We gather  some known properties and results on the module $H^0_\mm(B)$ we are interested in. 
Notice that $H^0_\mm(B)$ is a $\NN$-graded $B$-module; for all $\nu \in \NN$, we will denote by $H^0_\mm(B)_\nu$ its $\nu^{\text{th}}$ graded part. 

\subsubsection{The duality}

As proved in \cite[\S 1.5]{Jou96}, for all $\nu >\delta$ we have $H^0_\mm(B)_\nu=0$.  If $\nu=\delta$ then it is shown that there exists an isomorphism of $A$-modules $A \simeq H^0_\mm(B)_\delta$. Moreover, for all $0\leq \nu \leq \delta$ the  multiplication map (defined by the $B$-module structure of $H^0_\mm(B)$) 
$$B_{\delta-\nu}\otimes_A H^0_\mm(B)_\nu \rightarrow H^0_\mm(B)_\delta:b\otimes b' \mapsto bb'$$
induces the isomorphism of $A$-modules
\begin{equation}\label{dual1}
H^0_\mm(B)_\nu \xrightarrow{\sim} \Hom_A(B_{\delta-\nu},H^0_\mm(B)_\delta): b \mapsto \left(c \mapsto b\otimes c \right).
\end{equation}
This duality (recall that $H^0_\mm(B)_\delta\simeq A$) has been made explicit by Jouanolou by using the \emph{Morley forms} \cite[\S 3.11]{Jou97}. Before describing these forms, we recall the construction of the \emph{Sylvester forms} \cite[\S 3.10]{Jou97} which are the first examples of simple non-trivial inertia forms.

\subsubsection{Sylvester forms}

Suppose given $\beta=(\beta_1,\beta_2) \in \NN\times \NN$ such that $|\beta|=\beta_1+\beta_2\leq d_1-1$. Then, $f_1$ and $f_2$ can be decomposed in $C$ as
\begin{align}\label{sylv-decomp}
f_1(X_1,X_2)=X_1^{\beta_1+1}f_{1,1}+X_2^{\beta_2+1}f_{1,2}, \ \  f_2(X_1,X_2)=X_1^{\beta_1+1}f_{2,1}+X_2^{\beta_2+1}f_{2,2}
\end{align}
where $f_{i,j} \in C$ is homogeneous of degree $d_i-\beta_j-1$. Therefore the determinant of the matrix $(f_{i,j})_{i,j=1,2}$ is a homogeneous polynomial in $C$ of degree $\delta-|\beta|$. The class of this determinant in $B$ turns out to be independent of the choice of the decompositions \eqref{sylv-decomp}; it is a called a \emph{Sylvester form} of $f_1,f_2$ and will be denoted $\sylv_\beta(f_1,f_2) \in B_{\delta-|\beta|}$.

It is easy to check that, for $k=1,2$, we have $X_k^{\beta_k+1}\det(f_{i,j})_{i,j=1,2} \in I$. We deduce that $\sylv_\beta(f_1,f_2)$ is an inertia form of $f_1,f_2$ w.r.t.~$X_1,X_2$, that is to say that
 $$\sylv_\beta(f_1,f_2)  \in H^0_\mm(B)_{\delta-|\beta|}.$$

The Sylvester form $\sylv_{(0,0)}(f_1,f_2)$, very similar to the classical Jacobian, plays a particular r\^ole since it makes explicit the isomorphism 
\begin{equation}\label{Diso}
A \xrightarrow{\sim} H^0_\mm(B)_\delta : a \mapsto a.\sylv_{(0,0)}(f_1,f_2).
\end{equation}
Moreover, by duality we deduce that for all $\alpha,\beta$  such that $0\leq |\alpha|=|\beta|\leq d_1-1$ we have
$$X^\alpha \sylv_\beta(f_1,f_2)=
\begin{cases}
 \sylv_{(0,0)}(f_1,f_2) & \text{ if } \alpha=\beta \\
0 & \text{ otherwise } 
\end{cases}$$
in $H^0_\mm(B)_\delta.$

\subsubsection{Morley forms}\label{sec:morley}

Introducing two new indeterminates $Y_1$ and $Y_2$, we choose arbitrary decompositions in the polynomial ring $A[X_1,X_2,Y_1,Y_2]$:
\begin{align}\label{morley-decomp}
f_i(X_1,X_2)-f_i(Y_1,Y_2) &= (X_1-Y_1)h_{i,1}+(X_2-Y_2)h_{i,2}, \ \ i=1,2.
\end{align}
The determinant of the matrix $(h_{i,j})_{i,j=1,2}$ is a polynomial in $A[X_1,X_2,Y_1,Y_2]$ which is easily seen to be homogeneous in the variables $X_1,X_2$, resp.~$Y_1,Y_2$, of degree $\delta$. Now, consider the ring
$$B\otimes_A B \simeq \frac{A[X_1,X_2,Y_1,Y_2]}{(f_1(X_1,X_2),f_1(Y_1,Y_2),f_2(X_1,X_2),f_2(Y_1,Y_2))}.$$
Since $B$ is graded, $B \otimes_A B$ inherits of a canonical bi-grading (a grading w.r.t.~the variables $X_1,X_2$ and another one w.r.t.~$Y_1,Y_2$); we set
$$B\otimes_AB = \bigoplus_{(p,q)\in \NN\times \NN} B_p \otimes_A B_q = \bigoplus_{r\in \NN} \left( \bigoplus_{p+q=r} B_p\otimes_A B_q \right)=\bigoplus_{r\in \NN}[B\otimes_A B]_r.$$
The \emph{Morley form} of $f_1$ and $f_2$, denoted $\morl(f_1,f_2)$, is the class of $\det(h_{i,j})_{i,j=1,2}$ in $B\otimes_A B$. It is independent of the choice of the decompositions \eqref{morley-decomp}. Since $\morl(f_1,f_2) \in [B\otimes_AB]_\delta$, for all $(p,q) \in \NN\times \NN$ such that $p+q=\delta$ we denote by $\morl_{p,q}(f_1,f_2) \in B_p\otimes_k B_q$ its homogeneous component of bi-degree $(p,q)$.

Recall some properties of Morley forms that will be useful in the rest of this paper (see \cite[\S 3.11]{Jou97} for the proofs):
\begin{enumerate}
 \item Let $\tau$ be the symmetry exchanging $X_i$ with $Y_i$ for $i=1,2$. Then $\tau$ leaves $\morl(f_1,f_2)$ invariant and for all pairs $(p,q) \in \NN\times \NN$ such that  $p+q=\delta$ we have $\tau(\morl_{p,q}(f_1,f_2))=\morl_{q,p}(f_1,f_2)$.
\item $\morl_{\delta,0}(f_1,f_2)=\sylv_{(0,0)}(f_1,f_2)\otimes 1 \in B_\delta\otimes_A B_0$ and hence, by the above property, $\morl_{0,\delta}(f_1,f_2)=1 \otimes\sylv_{(0,0)}(f_1,f_2) \in B_0\otimes_A B_{\delta}$. 
\item The fact that $(X_i-Y_i)\morl(f_1,f_2)=0 \in [B\otimes_A B]_{\delta+1}$ implies that
\begin{align*}
 (b\otimes 1)\morl_{p,q}(f_1,f_2)&= \sylv_{(0,0)}(f_1,f_2)\otimes b && \text{ for all } b \in B_q \\
(1\otimes b) \morl_{p,q}(f_1,f_2)&= b\otimes \sylv_{(0,0)}(f_1,f_2) && \text{ for all } b \in B_p
\end{align*}

\end{enumerate}

\subsubsection{Explicit duality}

For all $0 \leq \nu \leq \delta$, choosing an arbitrary decomposition
$$\morl_{\delta-\nu,\nu}(f_1,f_2)=\sum_{s} x_s\otimes y_s \ \text{ with } \ x_s \in B_{\delta-\nu}, \, y_s \in B_\nu,$$
we have the following isomorphism of $A$-modules \cite[\S 3.6]{Jou96} (see also \cite{Jou07}) 
\begin{eqnarray}\label{theta-iso}
 \theta_\nu:  {B_{\delta-\nu}}^{\vee}=\Hom_A(B_{\delta-\nu},A) & \xrightarrow{\sim} & H^0_\mm(B)_\nu \subset B_\nu \\ \nonumber
u & \mapsto & \sum_{s} u(x_s)y_s.
\end{eqnarray}
In particular, this isomorphism shows that it is possible to describe explicitly all the inertia forms of $f_1,f_2$ of degree $\nu$ if one can describe explicitly the dual of $B_{\delta-\nu}$. This is the approach we will follow hereafter. For technical reasons, our analysis is divided into the three intervals $0\leq \nu \leq d_1-2$, $d_1-1\leq \nu \leq d_2-2$ and $d_2-1\leq \nu \leq \delta$. Of course, depending on the values of $d_1$ and $d_2$, it may happen that one or two of these intervals are  empty. We recall that we always assume that $1\leq d_1\leq d_2$.

\subsection{Inertia forms of degree $\leq d_1-2$}\label{sec:cas3} Assuming that $2\leq d_1 \leq d_2$, we analyze the inertia forms of $f_1,f_2$ of degree $\nu$ such that $0\leq \nu \leq d_1-2$, or equivalently such that $d_2\leq \delta-\nu \leq \delta$. 

For all integers $\nu=0,\ldots,d_1-1$, we introduce the new indeterminates 
$$\ud{W}=(W_0,\ldots,W_{\delta-2\nu+1}), \ \ud{T}=(T_0,\ldots,T_{\delta-\nu})$$ 
and we consider the polynomials 
\begin{align*}
h &= W_0X_1^{\delta-2\nu+1}+\cdots+W_{\delta-2\nu+1}X_2^{\delta-2\nu+1}, \\ 
\varphi &= T_0X_1^{\delta-\nu}+\cdots+T_{\delta-\nu}X_2^{\delta-\nu}.
\end{align*}  
We define the determinant $D(\nu)$ as
$$
\underbrace{\left|\begin{array}{ccc}
U_0 &  & 0 \\
 & \ddots &  \\
\vdots & &  U_0 \\
  & & \\ 
U_{d_1} & & \vdots \\
 & \ddots & \\
0 & & U_{d_1}\\
0 & \cdots & 0
\end{array}\right.
}_{d_2-\nu}
\underbrace{\left.\begin{array}{ccc}
V_0 &  & 0 \\
 & \ddots &  \\
\vdots & &  V_0 \\
  & & \\ 
V_{d_2} & & \vdots \\
 & \ddots & \\
0 & & V_{d_2}\\
0 & \cdots & 0
\end{array}\right.
}_{d_1-\nu}
\underbrace{\left.\begin{array}{ccc}
W_0 &  & 0 \\
 & \ddots &  \\
\vdots & &  W_0 \\
  & & \\ 
W_{\delta-2\nu+1} & & \vdots \\
 & \ddots & \\
0 & & W_{\delta-2\nu+1}\\
X_1^{\nu} & \cdots & X_2^{\nu}
\end{array}\right|
}_{\nu+1}
$$
and the determinant $D_1(\nu)$ as
$$
\underbrace{\left|\begin{array}{ccc}
U_0 &  & 0 \\
 & \ddots &  \\
\vdots & &  U_0 \\
  & & \\ 
U_{d_1} & & \vdots \\
 & \ddots & \\
0 & & U_{d_1}
\end{array}\right.
}_{d_2-\nu-1}
\underbrace{\left.\begin{array}{ccc}
V_0 &  & 0 \\
 & \ddots &  \\
\vdots & &  V_0 \\
  & & \\ 
V_{d_2} & & \vdots \\
 & \ddots & \\
0 & & V_{d_2}
\end{array}\right.
}_{d_1-\nu-1}
\begin{array}{c}
 T_0 \\
 \\
\vdots \\
 \\
\\
\vdots \\
\\
T_{\delta-\nu}
\end{array}
\underbrace{\left.\begin{array}{ccc}
W_0 &  & 0 \\
 & \ddots &  \\
\vdots & &  W_0 \\
  & & \\ 
W_{\delta-2\nu+1} & & \vdots \\
 & \ddots & \\
0 & & W_{\delta-2\nu+1}
\end{array}\right|
}_{\nu}
$$
Notice that $D(0)=\Res(f_1,f_2)$ and $D_1(0)=\sum_{i=0}^\delta \SRes_{\delta-i}T_i$. 
\begin{lemma}[$1\leq \nu \leq d_1-1$]\label{lem-cas3} 
In $C[\underline{W},\underline{T}]/(f_1(X),f_2(X))$ we have 
$$\varphi D(\nu)=(-1)^{d_1+1}D_1(\nu) \, \sylv_{(0,0)}(f_1,f_2).$$
\end{lemma}
\begin{proof} By specialization, it is sufficient to prove the claimed equality in the generic case, that is to say in the case where $A=\ZZ[U_0,\ldots,U_{d_1},V_0,\ldots,V_{d_2}]$. The proof of this lemma is a straightforward extension of the proof of \cite[Lemme 3.11.18.30]{Jou97} that we closely follow.

Denote by $(r_i)_{i=0,\ldots,\delta-\nu+1},r$ the rows of the matrix defining the determinant $D(\nu)$ from top to bottom. By construction, we have
\begin{multline}\label{eq1-propcas2}
\left(\sum_{i=0}^{\delta-\nu+1} X_1^{\delta-\nu+1-i}X_2^{i}r_i \right) - hr = \\
\left( X_1^{d_2-\nu-1}f_1,\ldots,X_2^{d_2-\nu-1}f_1,X_1^{d_1-\nu-1}f_2,\ldots,X_2^{d_1-\nu-1}f_2,0,\ldots,0\right)
\end{multline}
which shows that $X_1^{\delta-\nu+1}D(\nu) \in (f,g)$ and $X_2^{\delta-\nu+1}D(\nu) \in (f,g)$, i.e.~that $D(\nu)$ is an inertia form of $f_1,f_2$ w.r.t.~$\mm$ in $C[\ud{W},\ud{T}]$. Notice that $D(\nu)$ is homogeneous of degree $\nu$ in the variables $X_1,X_2$. 

Since $(f_1,f_2)$ is a $C$-regular sequence, we have  the canonical isomorphism, inverse of \eqref{Diso},
\begin{equation}\label{lambda}
\lambda : H^0_\mm(C[\underline{W},\underline{T}]/(f_1,f_2))_\delta \xrightarrow{\sim} A[\underline{W},\underline{T}]
\end{equation}
with the property that 
$\lambda(b)\, \sylv_0(f_1,f_2)=b$ for all $b \in  H^0_\mm(C[\underline{W},\underline{T}]/(f_1,f_2))_\delta$. 
Now, consider the $A[\ud{W},\ud{T}]$-linear maps
$$ \Lambda : A[\ud{W},\ud{T}][X_1,X_2]_{\delta-\nu} \rightarrow A[\ud{W},\ud{T}] : F \mapsto \lambda(FD(\nu))$$
and $$v: A[\ud{W},\ud{T}]^{\delta-\nu} \rightarrow A[\ud{W},\ud{T}]^{\delta-\nu+1}$$ which is given by the matrix defining $D_1(\nu)$ after deleting its unique column depending on the $T_i$'s. It is clear that $\Lambda$ vanishes on $(f_1,f_2)_{\delta-\nu}$, that corresponds to the first $\delta-2\nu$ columns of the matrix of $D_1(\nu)$, and on $(h)_{\delta-\nu}$, that corresponds to the last $\nu$ columns of the matrix of $D_1(\nu)$ since  $hD(\nu) \in (f,g)$ by \eqref{eq1-propcas2}. Therefore,  $\Lambda \circ v=0$ and hence $\Lambda$ belongs to the kernel of $v^\vee$, the dual of $v$. Moreover, it is not hard to check that the depth of the ideal of $(\delta-\nu)$-minors of $v$ is at least 2 and this implies that the Buchsbaum-Rim complex associated to $v^\vee$ is acyclic; it is of the form
$$ 0 \rightarrow  A[\ud{W},\ud{T}]^\vee \simeq \bigwedge^{\delta-\nu+1}(A[\ud{W},\ud{T}]^{\delta-\nu+1})^\vee  \rightarrow (A[\ud{W},\ud{T}]^{\delta-\nu+1})^\vee \xrightarrow{v^\vee} (A[\ud{W},\ud{T}]^{\delta-\nu})^\vee.$$
It follows that there exists an element $a \in A[\ud{W},\ud{T}]$ such that, for all elements $F \in A[\ud{W},\ud{T}][X_1,X_2]_{\delta-\nu}$ we have
\begin{multline*}
\Lambda(F)= a \det\left( X_1^{d_2-\nu-2}f_1,\ldots,X_2^{d_2-\nu-2}f_1,\right. \\ \left. X_1^{d_1-\nu-2}f_2,\ldots,X_2^{d_1-\nu-2}f_2,F,X_1^{\nu-1} h,\ldots,X_2^{\nu-1} h\right)
\end{multline*}
(notice that $X_1^{\nu-1} h,\ldots,X_2^{\nu-1} h$ disappear in the case $\nu=0$).
In particular, $\lambda(\varphi D(\nu))=aD_1(\nu)$ in $A[\ud{W},\ud{T}]$, that is to say
$$\varphi D(\nu) = a D_1(\nu) \sylv_{(0,0)}(f_1,f_2) \text{ in } \left(A[\ud{W},\ud{T}][X_1,X_2]/(f_1,f_2)\right)_\delta.$$
By inspecting the homogeneous degrees w.r.t.~$\ud{U},\ud{V},\ud{T},\ud{W}$, we deduce that $a\in \ZZ$. Then, to determine $a$ we consider the specialization
$$f_1 \mapsto X_1^{d_1}, \ f_2 \mapsto X_2^{d_2}, \ h \mapsto X_1^{d_1-\nu-1}X_2^{d_2-\nu}, \ \varphi \mapsto X_1^{d_1-1}X_2^{d_2-\nu-1}$$
which sends $D(\nu)$ to $(-1)^{\nu(d_1-\nu)}X_2^{\nu}$ and  $D_1(\nu)$ to $(-1)^{(d_1-\nu-1)(\nu+1)}$. We conclude that 
$$a=(-1)^{\nu(d_1-\nu)-(d_1-\nu-1)(\nu+1)}=(-1)^{2\nu-d_1+1}=(-1)^{d_1+1}$$
 and the lemma is proved.
\end{proof}

It is interesting to notice that in the case $\nu=0$, this lemma shows that
$$\left(\sum_{i=0}^{\delta}T_i X_1^{\delta-i}X_2^i\right)\Res(f_1,f_2)=(-1)^{d_1+1}\left( \sum_{i=0}^{\delta}  \SRes_{\delta-i}(f_1,f_2) T_i\right)\sylv_{(0,0)}(f_1,f_2)$$
in $C[\underline{T}]/(f_1(X),f_2(X))$. Writing $\sylv_{(0,0)}(f_1,f_2)=\sum_{i=0}^{\delta} q_i X_1^{\delta-i}X_2^i$ with $q_i \in A$, we deduce by specialization of $T_i$ to $q_i$ for all $i=0,\ldots,\delta$, and using the isomorphism \eqref{lambda}, that  
\begin{align}\label{Res=SRes}
\Res(f_1,f_2)=(-1)^{d_1+1} \sum_{i=0}^{\delta}  \SRes_{\delta-i}(f_1,f_2) q_i
\end{align}
in $A$. Observe that this is also equal to $(-1)^{d_1+1}$ times the determinant $D_1(0)$ where $T_i$ is specialized to $q_i$ for all $i=0,\ldots,\delta$.

Given two free modules $F,G$ and a linear map $u:F\rightarrow G$, we will denote by $\Det_q(u)$ the determinantal ideal generated by the $q$-minors of $u$. 

\begin{theorem}[$0\leq \nu \leq d_1-2$]\label{prop-cas3}
Let $D(\nu)=\sum_{|\beta|=\nu}D_\beta(X_1,X_2)W^\beta$. The element $D_\beta \in C$ is an inertia form of degree $\nu$ for all $\beta$ such that $|\beta|=\nu$.

Moreover, if the inequality
\begin{equation}\label{prof2-cas3}
\mathrm{depth}_{A}\Det_{\delta-2\nu}\left(C_{\delta-\nu-d_1}\oplus C_{\delta-\nu-d_2} \xrightarrow{\fb=(f_1,f_2)} C_{\delta-\nu} \right) \geq \nu+2
\end{equation}
holds, then the collection of inertia forms $(D_\beta)_{|\beta|=\nu}$ is a system of generators of $H^0_\mm(B)_{\nu}$. 
\end{theorem}
\begin{proof} The formula \eqref{eq1-propcas2} shows that 
$\mm^{\delta-\nu+1}D(\nu) \subset (f_1,f_2)$ in $C[\ud{W}]$
and hence that $D_\beta(X_1,X_2) \in C$ is an inertia form of degree $\nu$ of $f_1,f_2$ w.r.t.~$\mm$ for all $\beta$ such that $|\beta|=\nu$.

Now, let $G \in C_{\nu}$ be an inertia form of $f_1,f_2$ w.r.t.~$\mm$ of degree $\nu$ and consider the $A$-linear map
$$ \Gamma_G : C_{\delta-\nu} \rightarrow A : F \mapsto \lambda(FG)$$
where $\lambda$ is defined by \eqref{lambda}. It is clear that $\Gamma_G$ vanishes on $(f_1,f_2)_{\delta-\nu}$, that is to say that $\Gamma_G \circ\fb=0$, and hence that $\Gamma_G$ belongs to the kernel of the dual $\fb^\vee$ of $\fb$. Notice that the first $\delta-2\nu$ columns of the matrix of $D_1(\nu)$ give a matrix of $\fb$ in appropriate monomial bases. Under the hypothesis \eqref{prof2-cas3}, the Buchsbaum-Rim complex associated to $\fb^\vee$ is acyclic; it is of the form
$$ \cdots \rightarrow \bigwedge^{\delta-2\nu+1}(C_{\delta-\nu})^\vee  \xrightarrow{\varepsilon_\nu} (C_{\delta-\nu})^\vee \xrightarrow{\fb^\vee} (C_{\delta-\nu-d_1})^\vee\oplus(C_{\delta-\nu-d_2})^\vee$$
where we recall that the map $\varepsilon_\nu$ sends the basis element 
$$(X^{\alpha_1})^\vee \wedge \cdots \wedge (X^{\alpha_m})^\vee$$ 
where $m=\delta-2\nu+1$, to
\begin{multline*}
\sum_{i=1}^{m}(-1)^i (X^{\alpha_i})^\vee \times \\
\left( {\fb}^\vee((X^{\alpha_1})^\vee) \wedge \cdots \wedge {\fb}^\vee((X^{\alpha_{i-1}})^\vee) \wedge {\fb}^\vee((X^{\alpha_{i+1}})^\vee) \wedge \cdots \wedge {\fb}^\vee((X^{\alpha_m})^\vee) \right).
\end{multline*}
For all set $I \subset \{1,\ldots,\delta-\nu+1\}$ with cardinality $|I|=\delta-2\nu+1$, we denote by $D_I \in A[\ud{T}]$ the determinant of the minor of $D_1(\nu)$ corresponding to the first $\delta-2\nu+1$ columns and the rows indexed by $I$. From the acyclicity of the above complex we deduce that there exists a collection of elements $a_I \in A$ such that
\begin{equation}\label{eq2-proofcas3}
\lambda(\varphi G)=\sum_{\substack{I \subset \{1,\ldots,\delta-\nu+1\} \\ |I|=\delta-2\nu+1}} a_ID_I \ \in A[\ud{T}].
\end{equation}
To finish the proof we distinguish the two cases $\nu=0$ and $1\leq \nu \leq d_1-2$.

If $\nu=0$, \eqref{eq2-proofcas3} reduces to the equality $\lambda(\varphi G)= a D_1(\nu)$ for some $a \in A$, since $D_1(\nu)$ does not depend on $\ud{W}$ in this case. Specializing $\varphi$ to $\sylv_{(0,0)}(f_1,f_2)$, we deduce that
$$G=G\lambda(\sylv_{(0,0)}(f_1,f_2))=\lambda(\sylv_{(0,0)}(f_1,f_2) G)=a\Res(f_1,f_2)$$
because $G\in A$ and this specialization sends $D_1(\nu)$ to $\Res(f_1,f_2)$ \cite[Corollaire 3.10.22]{Jou97}. Consequently, as $\Res(f_1,f_2)=D(0)=D_{\beta=(0,0)}$ in this case, we have $G=aD(0)$, i.e.~$D(0)$ is a generator of $H^0_\mm(B)_0$.

Now, assume that $1 \leq \nu \leq d_1-2$. By definition of $(D_I)_{I}$, there exists a collection of polynomials $q_I(W) \in \ZZ[\ud{W}]$ such that
$$(-1)^{d_1+1}D_1(\nu)=\sum_{\substack{I \subset \{1,\ldots,\delta-\nu+1\} \\ |I|=\delta-2\nu+1}} D_I \, q_I(\ud{W}) \ \in A[\ud{T},\ud{W}].$$
The collection $(q_I(\ud{W}))$ form a basis of the homogeneous polynomials of degree $\nu$ in the variables $\ud{W}$, as well as the collection $(\ud{W}^\beta)_{|\beta|=\nu}$ by \cite[Remarque 3.11.18.22]{Jou97}. Therefore, there exist polynomials $H_I \in A[X_1,X_2]_\nu$ such that 
$$ D(\nu) = \sum_{\substack{I \subset \{1,\ldots,\delta-\nu+1\} \\ |I|=\delta-2\nu+1}} H_I \, q_I(\ud{W}) \ \in A[\ud{W}]$$
and such that $(D_\beta)_{|\beta|=\nu}$ and $(H_I)_{|I|=\delta-2\nu+1}$ form two systems of generators of the same $A$-module. With this notation, Lemma \ref{lem-cas3}  gives
$$\sum_{\substack{I \subset \{1,\ldots,\delta-\nu+1\} \\ |I|=\delta-2\nu+1}} \lambda(\varphi H_I) \, q_I(\ud{W}) = \sum_{\substack{I \subset \{1,\ldots,\delta-\nu+1\} \\ |I|=\delta-2\nu+1}} D_I \, q_I(\ud{W}),$$
i.e.~$\lambda(\varphi H_I)=D_I$ for all $I \subset \{1,\ldots,\delta-\nu+1\}$ with $|I|=\delta-2\nu+1$.
Therefore, we deduce from \eqref{eq2-proofcas3} that
$$ \lambda (\varphi G) = \lambda \left(\varphi \left( \sum_{I}a_IH_I \right) \right).$$
Now, since $\lambda$ is injective, we have $ \varphi G=\varphi \left( \sum_{I}a_IH_I \right)$ in $H^0_\mm(B)_\delta$. By identifying the coefficients in the variables $\ud{T}$, we deduce that the two multiplication maps by $G$ and $\sum_I a_IH_I$ from $B_{\delta-\nu}$ to $H^0_\mm(B)_\delta$ coincide. Therefore, by the duality \eqref{dual1} we deduce that $G=\sum_I a_IH_I$.
\end{proof}

\begin{remark}
In the generic case, i.e.~$A=\ZZ[U_0,\ldots,U_{d_1},V_0,\ldots,V_{d_2}]$, one can show very similarly to \cite[Proposition 3.11.18.19 (c)]{Jou97}, that the inequality \eqref{prof2-cas3} holds.
\end{remark}

Theorem \ref{prop-cas3} shows that $D(0)=\Res(f_1,f_2)$ is a generator of $H^0_\mm(B)_0$ if the ideal of $A$ generated by the subresultants $\SRes_i(f_1,f_2)$, $i=0,\ldots,\delta$, has depth at least 2. To show that  this condition can not be avoided we consider the following example: $A=\ZZ[U_0,V_{2}]$, $f_1=U_0X_1^{d_1}$ and $f_2=V_{d_2}X_2^{d_2}$. It is easy to compute that  $\Res(f_1,f_2)=U_0^{d_2}V_{d_2}^{d_1}$ and $H^0_\mm(B)_0=(U_0V_{d_2}) \subset A$. Therefore, $\Res(f_1,f_2)$ is not a generator of $H^0_\mm(B)_0$. Also, one can check that $\SRes_{d_1-1}(f_1,f_2)=U_0^{d_2-1}V_{d_2}^{d_1-1}$ and that all the other subresultants vanish, so the ideal generated by all the subresultants has depth exactly 1.

\medskip

Later on we will be concerned with the case $\nu=1$, so we describe in more detail the situation in this case. Notice that this case may occur only if $d_2\geq d_1\geq 3$.
For simplicity, we rename the inertia forms $D_\beta(p,q)$ with $|\beta|=1$ as $D_{i}(p,q)$, $i=0,\ldots,\delta-1$ in the following way~:
{\small $$
D(1)=\underbrace{\left|\begin{array}{ccc}
U_0 &  & 0 \\
 & \ddots &  \\
\vdots & &  U_0 \\
  & & \\ 
U_{\mu} & & \vdots \\
 & \ddots & \\
0 & & U_{\mu}\\
0 & \cdots & 0
\end{array}\right.
}_{d-\mu-1}
\underbrace{\left.\begin{array}{ccc}
V_0 &  & 0 \\
 & \ddots &  \\
\vdots & &  V_0 \\
  & & \\ 
V_{d-\mu} & & \vdots \\
 & \ddots & \\
0 & & V_{d-\mu}\\
0 & \cdots & 0
\end{array}\right.
}_{\mu-1}
\underbrace{\left.\begin{array}{ccc}
W_0 & 0 \\
  &  \\
\vdots &  W_0 \\
  &  \\ 
& \\
W_{\delta-1} &  \vdots \\
  & \\
0 &  W_{\delta-1}\\
X_1 & X_2
\end{array}\right|
}_{2}=\sum_{i=0}^{\delta-1}D_{i}(X_1,X_2)W_i.
$$}
\begin{corollary}\label{nu=1:cas3} With the above notation, for all $i=0,\ldots,\delta-1$ we have, in $A$,
$$D_i(p,q)=X_2\SRes_{\delta-i}(p,q)-X_1\SRes_{\delta-i-1}(p,q).$$
Moreover, if the inequality
\begin{equation*}
\mathrm{depth}_{A}\Det_{\delta-2}
\underbrace{\left(\begin{array}{ccc}
U_0 &  & 0 \\
 & \ddots &  \\
\vdots & &  U_0 \\
  & & \\ 
U_{d_1} & & \vdots \\
 & \ddots & \\
0 & & U_{d_1}
\end{array}\right.
}_{d_2-2}
\underbrace{\left.\begin{array}{ccc}
V_0 &  & 0 \\
 & \ddots &  \\
\vdots & &  V_0 \\
  & & \\ 
V_{d_2} & & \vdots \\
 & \ddots & \\
0 & & V_{d_2}
\end{array}\right)
}_{d_1-2}
 \geq \nu+2
\end{equation*}
holds, then the collection of inertia forms $(D_i)_{i=0,\ldots,\delta-1}$ is a system of generators of $H^0_\mm(B)_{1}$. 
\end{corollary}
\begin{proof}
This is a straightforward computation from the definitions. 
\end{proof}

\subsection{Inertia forms of degree $\geq d_1 -1$ and $\leq d_2-2$}\label{Mnu}
 Choosing a decomposition \eqref{morley-decomp}, we set
$$\det\left(h_{i,j}(X_1,X_2,Y_1,Y_2)\right)_{i,j=1,2}=\sum_{0\leq |\beta|\leq \delta}  q_{\beta}(X)Y^\beta.$$
For all integers $\nu$ such that $ d_1-1 \leq \nu \leq d_2-2$, we have $d_1\leq \delta-\nu \leq d_2-1$ and we consider the $(\delta-\nu+1)\times(\delta-\nu-d_1+2)$-matrix
$$
M_\nu=\underbrace{\left(\begin{array}{ccc}
U_0 &  & 0 \\
 & \ddots &  \\
\vdots & &  U_0 \\
  & & \\ 
U_{d_1} & & \vdots \\
 & \ddots & \\
 & & U_{d_1}
\end{array}\right.
}_{\delta-\nu-d_1+1}
\underbrace{\left.\begin{array}{c}
 \\
\vdots \\
 \\
q_\beta(X) \\
\\
\vdots \\
\\
\\
\end{array}\right)}_{1}
$$
built as follows: the left block of $M_\nu$ is the matrix of the multiplication map $C_{\delta-\nu-d_1} \rightarrow C_{\delta-\nu}: p \mapsto pf_1$ in the monomial bases ordered with the lexicographical order $Y_1 \succ Y_2$; the last column contains the coefficients of $\sum_{|\beta|=\delta-\nu} q_{\beta}(X)Y^\beta $ in the same ordered monomial basis $(Y^\beta)_{|\beta|=\delta-\nu}$. 

Set $m=\delta-\nu-d_1+2$ and denote by $\Delta_{\alpha_1,\ldots,\alpha_m}$, with $Y^{\alpha_1} \succ \cdots \succ Y^{\alpha_m}$ and $|\alpha_i|=\delta-\nu$ for all $i=1,\ldots,m$, the determinant of the $m$-minor of $M_\nu$ corresponding to the rows of $M_\nu$ indexed by $(Y^{\alpha_i})_{i=1,\ldots,m}$. 

\begin{theorem}[$d_1-1\leq \nu \leq d_2-2 $]\label{prop-cas2}
The minors $\Delta_{\alpha_1,\ldots,\alpha_m} \in A[X_1,X_2]$ are independent of the choice of the decomposition \eqref{morley-decomp} modulo the ideal $(f_1,f_2)$ and are inertia forms of $f_1,f_2$ w.r.t.~$\mm$ of degree $\nu=|\alpha_1|=\cdots=|\alpha_m|$.

Moreover, if the inequality 
\begin{equation}\label{depth-cas2}
\mathrm{depth}_A \left( U_0,\ldots,U_{d_1}  \right)\geq d_1+1 
\end{equation}
holds, then the collection of minors $\left( \Delta_{\alpha_1,\ldots,\alpha_m} \right)_{Y^{\alpha_1} \succ \cdots \succ Y^{\alpha_m}}$ is a system of generators of the $A$-module $H^0_\mm(B)_\nu$.
\end{theorem}

\begin{proof} Let $(h_{i,j})_{i,j=1,2}$ and $(h_{i,j}')_{i,j=1,2}$ be two decompositions \eqref{morley-decomp} and set
$$ \det(h_{i,j})=\sum_{|\beta|\leq \delta} q_\beta(X)Y^\beta, \ \ \det(h_{i,j}')=\sum_{|\beta|\leq \delta} q'_\beta(X)Y^\beta.$$
For any choice of sequences $(\alpha_1,\ldots,\alpha_m)$ such that $Y^{\alpha_1} \succ \cdots \succ Y^{\alpha_m}$ and $|\alpha_i|=\delta-\nu$ for all $i=1,\ldots,m$, we will denote  by $\Delta_{\alpha_1,\ldots,\alpha_m}$ and $\Delta'_{\alpha_1,\ldots,\alpha_m}$ the determinants associated to the decompositions $(h_{i,j})_{i,j=1,2}$ and $(h_{i,j}')_{i,j=1,2}$ respectively.

By Section \ref{sec:morley}, we know that
$$\det(h_{i,j})-\det(h'_{i,j})= \sum_{|\beta|\leq \delta}(q_\beta(X)-q'_{\beta}(X))Y^\beta \in (f_1(Y),f_2(Y))\frac{A[X,Y]}{(f_1(X),f_2(X))}$$
and, by taking homogeneous components for $0\leq \nu \leq \delta$, that
\begin{equation}\label{pf:incl1}
\sum_{|\beta|=\delta-\nu}(q_\beta(X)-q'_{\beta}(X))Y^\beta \in (f_1(Y),f_2(Y))_{\delta-\nu}B[Y].
 \end{equation}
Assume now that $d_1\leq \delta-\nu \leq d_2-1$. Since $f_1(Y)$ is not a zero-divisor in $B[Y]$, we have the exact sequence 
\begin{equation}\label{exseq-f1}
0 \rightarrow B[Y]_{\delta-\nu-d_1} \xrightarrow{\times f_1} B[Y]_{\delta-\nu} \rightarrow \left(B[Y]/(f_1(Y))\right)_{\delta-\nu} \rightarrow 0
\end{equation}
and \eqref{pf:incl1} implies that we also have the exact sequence
 \begin{equation}\label{exseq-M}
B[Y]_{\delta-\nu-d_1}\oplus B \xrightarrow{\tilde{M}_\nu} B[Y]_{\delta-\nu} \rightarrow B[Y](f_1(Y))_{\delta-\nu} \rightarrow 0,
\end{equation}
$\tilde{M}_\nu$ being defined as the matrix $M_\nu$ where each element $q_\beta(X)$ in the last column is replaced by the difference $q_\beta(X)-q'_\beta(X)$ respectively.
Therefore, the comparison of  \eqref{exseq-f1} and \eqref{exseq-M} shows, by invariance of Fitting ideals, that the class of $\Delta_{\alpha_1,\ldots,\alpha_m}-\Delta'_{\alpha_1,\ldots,\alpha_m}$ in $B=A[X]/(f_1(X),f_2(X))$ is null. We deduce that  $\Delta_{\alpha_1,\ldots,\alpha_m}$ is independent of the choice of the decomposition \eqref{morley-decomp} modulo $(f_1,f_2)$, as claimed.

Since $d_1\leq \delta-\nu \leq d_2-1$, we have the exact sequence of $A$-modules
$$0 \rightarrow C_{\delta-\nu-d_1} \xrightarrow{\fb_1} C_{\delta-\nu} \rightarrow B_{\delta-\nu} \rightarrow 0,$$
where $\fb_1$ denotes the multiplication by $f_1$ and we deduce, by duality,
that we have the exact sequence of $A$-modules
$$ 0 \rightarrow {B_{\delta-\nu}}^\vee \rightarrow {C_{\delta-\nu}^\vee} \xrightarrow{{\fb_1}^\vee} {C_{\delta-\nu-d_1}}^\vee.$$
In particular, ${B_{\delta-\nu}}^\vee$ is isomorphic to the kernel of ${\fb_1}^\vee$. 

Now, consider the Buchsbaum-Rim complex  associated to ${\fb_1}^\vee$; it is of the form
\begin{align}\label{BRcpx}
 \cdots \rightarrow \bigwedge^{m}({C_{\delta-\nu}}^\vee) \xrightarrow{\varepsilon_\nu} {C_{\delta-\nu}}^\vee \xrightarrow{\fb^\vee} {C_{\delta-\nu-d_1}}^\vee \rightarrow 0
\end{align}
(recall $m=\delta-\nu-d_1+2$).
Since it is a complex, the image of $\varepsilon_{\nu}$ is contained in ${B_{\delta-\nu}}^\vee$ and hence we can consider the composition map
$$ \theta_\nu \circ \varepsilon_\nu : \bigwedge^{m}({C_{\delta-\nu}}^\vee) \rightarrow H^0_\mm(B)_\nu.$$
Choosing a decomposition \eqref{morley-decomp} and setting $\det(h_{i,j})_{i,j=1,2}=\sum_{|\beta|\leq \delta}q_\beta(X)Y^\beta$, we deduce that $\theta_\nu \circ \varepsilon_\nu$ sends the basis element $(Y^{\alpha_1})^\vee \wedge \cdots \wedge (Y^{\alpha_m})^\vee$, with $Y^{\alpha_1} \succ \cdots \succ Y^{\alpha_m}$, to the determinant $\Delta_{\alpha_1,\ldots,\alpha_m}$ up to sign. Therefore, all the determinants $\Delta_{\alpha_1,\ldots,\alpha_m}$ are inertia forms of degree $\nu$, as claimed. Moreover, if \eqref{depth-cas2} holds then \eqref{BRcpx} is acyclic since
$$\left( U_0,\ldots,U_{d_1} \right)^{\delta-\nu-d_1+1}=\Det_{\delta-\nu-d_1+1} \left( C_{\delta-\nu-d_1} \xrightarrow{\times f_1} C_{\delta-\nu} \right).$$
Therefore
$$\mathrm{depth}_A \left( \Det_{\delta-\nu-d_1+1} \left( C_{\delta-\nu-d_1} \xrightarrow{\times f_1} C_{\delta-\nu} \right) \right)\geq d_1+1$$
and it follows that the image of $\varepsilon_\nu$ is exactly $B_{\delta-\nu}$. We hence deduce that $\theta_\nu \circ \varepsilon_\nu$ is surjective.
\end{proof}

The case $d_1=1$ is particularly interesting because then the matrix $M_\nu$ is square for all $\nu$ (such that $d_1-1=0\leq \nu \leq d_2-2=\delta-1$). Therefore, if \eqref{depth-cas2} holds then $H^0_\mm(B)_\nu$ is a free $A$-module of rank 1 generated by $\det(M_\nu)$. In other words, for all $0\leq \nu \leq d_2-2$ we have the isomorphism 
$$ A \xrightarrow{\sim} H^0_\mm(B)_\nu : a \mapsto a\det(M_\nu).$$
For the sake of completeness, we give an alternate construction of a system of generators of $H^0_\mm(B)$ in this case. This system was discovered independently in \cite{HW} and \cite{HSV}. 

Since $d_1=1$, we have $f_1=U_0X_1+U_1X_2$. We know by \eqref{Diso} that the Sylvester form $h_\delta(f_1,f_2)=\sylv_{(0,0)}(f_1,f_2) \in H^0_\mm(B)_\delta$ is a generator of $H^0_\mm(B)_\delta$. Now, for all integer $i=1,\ldots,\delta$ we define $h_{\delta-i}$ by induction with  the formula 
$$h_{\delta-i}(f_1,f_2) = \sylv_{(0,0)}(f_1,h_{\delta-i+1}(f_1,f_2)) \in B_{\delta-i}.$$
\begin{proposition}[$d_1=1$, $0\leq \nu \leq d_2-2=\delta-1$]\label{cas2-d=1} With the above notation, $h_{\nu}(f_1,f_2)$ is an inertia form of $f_1,f_2$ of degree $\nu$ and  is equal to $\det(M_\nu)$ up to sign in $H^0_\mm(B)_{\nu}$.
In particular, if $\mathrm{depth}_A (U_0,U_1) \geq 2$ then $h_{\nu}$ is a generator of $H^0_\mm(B)_\nu$.
\end{proposition}

\begin{proof} By specialization, it is sufficient to prove the claimed equality in the generic case, so we assume that $A=\ZZ[U_0,U_1,V_0,\ldots,V_{d_2}]$. In this case,  $\mathrm{depth}_A (U_0,U_1) \geq 2$ and hence $\det(M_\nu)$ is a generator of $H^0_\mm(B)_\nu$. 
By construction, we have, for all $\nu=0,\ldots,\delta-1$,
$$(X_1,X_2)h_{\nu} \subset (f_1,h_{\nu+1}) \subset A[X_1,X_2].$$
Therefore, since $\sylv_{(0,0)}(f_1,f_2) \in H^0_\mm(B)$, we deduce that $h_\nu$, for all $\nu=0,\ldots,\delta$, is an inertia form of $f_1,f_2$. Moreover, $h_\nu$ is homogeneous of degree $\nu$ in $X_1,X_2$, hence $h_\nu \in H^0_\mm(B)_\nu$, and by construction $h_\nu$ and $\det(M_\nu)$ are both homogeneous of degree $\delta-\nu+1$, resp.~$1$, in the variables $U_0,U_1$, resp.~$V_0,\ldots,V_{d_2}$.
It follows that there exists $a_\nu \in \ZZ$ such that $h_\nu=a_\nu\det(M_\nu)$ for all $\nu=0,\ldots,\delta-1$. Finally, to prove that $a_\nu=\pm 1$ for all $\nu$ we observe that the specialization sending $f_1$ to $X_1$ and $f_2$ to $X_2^{d_2}$ sends $h_\nu$ to $X_2^{\nu}$, for all $\nu=0,\ldots,\delta-1$. The last statement of this corollary is contained in Theorem \ref{prop-cas2}.

Observe that, as a consequence of this proof, $\det(M_0)$ and $h_{0}$ are both equal to $\Res(f_1,f_2)$ up to sign.
\end{proof}

Going back to the general setting of this paragraph, we now examine in more detail the case $\nu=1$ to make a link with subresultants. 
This case may occur only if $d_1=1$ or $d_1=2$. Since we have already studied the case $d_1=1$ above, we concentrate on the case $d_1=2$. 

So we assume that $f_1=U_0X_1^2+U_1X_1X_2+U_2X_2^2.$ For simplicity, we rename the inertia forms of degree $\nu=1$, that is $\Delta_{\alpha_1,\ldots,\alpha_{\delta-1}}(f_1,f_2)$ with $|\alpha_i|=\delta-1=d_2-1$. In this case, we have to consider the 
maximal $(d_2-1)$-minors of the $(d_2-1)\times (d_2)-$matrix 
$$
M_1=\underbrace{\left(\begin{array}{ccc}
U_0 &  & 0 \\
 & \ddots &  \\
U_1 & &  U_0 \\
  & \ddots & \\ 
U_{2} & & U_1 \\
 & \ddots & \\
 & & U_{2}
\end{array}\right.
}_{d_2-2}
\underbrace{\left.\begin{array}{c}
 \\
\vdots \\
 \\
q_\beta(X) \\
\\
\vdots \\
\\
\end{array}\right).}_{1}
$$
We define the inertia forms $\Delta_{i}$, $i\in \{1,\ldots,d-2\}$ by the formula ($\delta=d_2$ here)
$$
\underbrace{\left|\begin{array}{cccc}
U_0 &  & 0  & \\
 & \ddots &  & \vdots \\
U_1 & &  U_0 & \\
  & \ddots & & q_\beta(X) \\ 
U_{2} & & U_1 & \\
 & \ddots & & \vdots\\
 & & U_{2} & \vspace{.25cm}
\end{array}\right.
}_{M_1}
\underbrace{\left.\begin{array}{c}
T_0 \\
 \\
\vdots \\
 \\
\vdots \\
 \\
T_{\delta-1}
\vspace{.25cm} \\
\end{array}\right|}_{1}=\sum_{i=0}^{\delta-1}\Delta_{\delta-i}T_i. 
$$

\begin{lemma}[$d_1=2, \nu=1$]\label{SRes:mu=2} 
For all $i=0,\ldots,\delta-1$, we have, in $A$,
$$\Delta_i=X_2\SRes_{i+1}(f_1,f_2)-X_1\SRes_i(f_1,f_2).$$ 
\end{lemma}
\begin{proof} We claim that this is a consequence of Lemma \ref{lem-cas3} applied with 
$$\varphi=\sum_{|\beta|=\delta-1} q_\beta(Y)X^\beta=\morl_{\delta-1,1}(f_1,f_2).$$
Indeed, denoting $D(X)$ the determinant $D(1)$ in this case, the properties of Morley's forms imply that
$$\varphi D(X)=D(X)\morl_{\delta-1,1}(f_1,f_2)=D(Y)\sylv_{(0,0)}(f_1,f_2)(X)$$
in  $C[\underline{W}][X,Y]/(f_1(X),f_2(X))(X)$.
Then, Lemma \ref{lem-cas3} shows that $$\varphi D(X)=-D_1(Y)\sylv_{(0,0)}(f_1,f_2),$$  so by comparison of these two equalities and duality it follows that 
$D(Y)=D_1(Y)$ in $C[\underline{W}][Y]$. Therefore,
$$\sum_{i=0}^{\delta-1}D_{i}(Y)W_i=D(Y)=D_1(Y)=\sum_{i=0}^{\delta-1}\Delta_{\delta-i}(Y)W_i$$
and hence $\Delta_i=D_{\delta-i}$ for all $i=0,\ldots,\delta-1$, as claimed.
\end{proof}

\subsection{Inertia forms of degree $\geq d_2-1$}
This last case is the easiest one. For all integers $\nu$ such that $d_2-1\leq \nu \leq \delta$, we have $0\leq \delta-\nu \leq d_1-1$. Since $I$ is generated in degree at least $d_1$, we have a canonical isomorphism $C_{\delta-\nu} \simeq B_{\delta-\nu}$ and hence ${C_{\delta-\nu}}^\vee \simeq {B_{\delta-\nu}}^\vee$. Therefore, the morphism \eqref{theta-iso} is completely explicit as it is easy to find a basis of the $A$-module ${C_{\delta-\nu}}^\vee$; for instance the dual of the monomial basis $(X^\alpha)_{|\alpha|=\delta-\nu}$ of $C_{\delta-\nu}$.

\begin{lemma}\label{Joulemme} For all $\nu \in \NN$ such that $0\leq \delta-\nu\leq d_1-1$, the following equality holds in $B_{\delta-\nu} \otimes_A B_\nu=C_{\delta-\nu} \otimes_A B_\nu$
$$ \morl_{\delta-\nu,\nu}(f_1,f_2)=\sum_{|\alpha|=\delta-\nu} X^\alpha\otimes \sylv_{\alpha}(f_1,f_2).$$
\end{lemma}
\begin{proof} The proof of \cite[Proposition 3.11.13]{Jou97} works verbatim.
\end{proof}

It follows that for all integers $\nu$ such that $d_2-1\leq \nu \leq \delta$, the isomorphism \eqref{theta-iso} is given by
\begin{eqnarray*}
 \theta_\nu : {C_{\delta-\nu}}^\vee & \xrightarrow{\sim} & H^0_\mm(B)_\nu \\
(X^\alpha)^\vee & \mapsto & \sylv_\alpha(f_1,f_2)
\end{eqnarray*}
and we have the
\begin{theorem}[$d_2-1 \leq \nu \leq \delta$]\label{easycase} The collection of all the Sylvester forms of degree $\delta-\nu$, that is $\left(\sylv_\alpha(f_1,f_2)\right)_{|\alpha|=\delta-\nu}$, yields a $A$-basis of $H^0_\mm(B)_\nu$. 
\end{theorem}

Finally, as we did in the previous sections we make explicit the case $\nu=1$ for later purposes. Here, the only interesting situation occurs when $d_1=d_2=2$ and we have
{\small \begin{align}\label{nu=1:d=2}
\sylv_{0,1}(f_1,f_2)&=
\left| 
\begin{array}{cc}
U_0 & U_2 \\
V_0 & V_2
 \end{array}
\right| X_1 +
\left| 
\begin{array}{cc}
U_1 & U_2 \\
V_1 & V_2
 \end{array}
\right| X_2=\SRes_2(f_1,f_2)X_2-\SRes_1(f_1,f_2)X_1,\\ \nonumber
\sylv_{1,0}(f_1,f_2)&=
\left| 
\begin{array}{cc}
U_0 & U_1 \\
V_0 & V_1
 \end{array}
\right| X_1+
\left| 
\begin{array}{cc}
U_0 & U_2 \\
V_0 & V_2
 \end{array}
\right| X_2=\SRes_0(f_1,f_2)X_1-\SRes_1(f_1,f_2)X_2
\end{align}
}

\section{Equations of the moving curve ideal}\label{eqRees}

We take again the parametrization \eqref{phi}
\begin{eqnarray*}
 \PP_\KK^1 & \xrightarrow{\phi} & \PP_\KK^2 \\
(X_1:X_2) & \mapsto & (g_1:g_2:g_3)(X_1,X_2).
\end{eqnarray*}
Without loss of generality, we will assume hereafter that \emph{the greatest common divisor of $g_1,g_2,g_3$ over $\KK[X_1,X_2]$ is a non-zero constant in $\KK$}. Moreover, we will restrict our study to the case of interest where the algebraic curve $\Cc$, image of $\phi$, has degree at least 2.

Let $p(\ud{X},\ud{T}),q(\ud{X},\ud{T})$ be a $\mu$-basis of the parametrization $\phi$, where $p$, resp.~$q$, has degree $\mu$, resp.~$d-\mu$, in the variables $X_1,X_2$. By Proposition \ref{pqTF}, the moving curve ideal of $\phi$ is equal to the ideal of inertia forms of $p,q$ with respect to the ideal $(X_1,X_2)$. Therefore, the results developed in Section \ref{inform} can be used to give some of the generators of the moving curve ideal of $\phi$, and sometimes a whole system of generators. To proceed, we set
\begin{align}\label{mbX}
p &= U_0(\underline{T})X_1^\mu+ \cdots+U_{\mu}(\underline{T})X_2^\mu, \\ \nonumber
q &= V_0(\underline{T})X_1^{d-\mu}+\cdots+V_{d-\mu}(\underline{T})X_2^{d-\mu},
\end{align}
where $U_0,\ldots,U_\mu$ and $V_0,\ldots,V_{d-\mu}$ are linear forms in  $A=\KK[T_1,T_2,T_3]$ and assume, without loss of generality, that $1\leq \mu\leq d-\mu$. 
We also define $C=A[X_1,X_2]$, $\mm=(X_1,X_2)$, $\delta=\mu+(d-\mu)-2=d-2$ and consider the graded quotient ring $B=C/(p,q)$.

\subsection{The case $\mu=1$}\label{sec:u=1}
From Section \ref{inform}, we have the following list of inertia forms of $p$ and $q$:
\begin{equation}\label{listTFcas1}
\begin{cases}
 \det(M_\nu(p,q)) \in H^0_\mm(B)_\nu &  \text{ for } 0\leq \nu \leq \delta-1,  \ \text{ see Theorem \ref{prop-cas2}},  \\
 \sylv_{(0,0)}(p,q) \in H^0_\mm(B)_\delta &  \text{ for } \nu=\delta, \ \text{ see Theorem \ref{easycase}}.  
\end{cases}
\end{equation}
It turns out that this collection of inertia forms always gives a system of generators of the moving curve ideal of $\phi$, as  conjectured in \cite[Conjecture 4.5]{HSV} and proved in \cite[Theorem 2.3]{HW}.

\begin{proposition}[$\mu=1$]  The two polynomials $p, q$ and the collection of inertia forms \eqref{listTFcas1} form a system of generators of the moving curve ideal of $\phi$.
\end{proposition}
\begin{proof} According to Theorem \ref{prop-cas2} and Theorem \ref{easycase}, we only have to prove that $\mathrm{depth}_A(U_0,U_1)\geq 2$. This inequality is a direct consequence of the facts that $\deg(\Cc)\geq 2$ and that $p$ is by definition a syzygy  of minimal degree of $g_1,g_2,g_3$.
\end{proof}

\subsection{The case $\mu=2$}  From Section \ref{inform}, we have the following list of inertia forms of $p$ and $q$:
\begin{equation}\label{listTFcas2}
\begin{cases}
 \Res(p,q) & \text{ for } \nu=0, \ \text{ see Theorem \ref{prop-cas3}, } \\ 
 \Delta_{\alpha_1,\ldots,\alpha_{\delta-\nu}}(p,q), |\alpha_i|=\delta-\nu &  \text{ for } 1\leq \nu \leq \delta-2, \ \text{ see Theorem \ref{prop-cas2}}, \\
 \sylv_{\alpha}(p,q), |\alpha|=\delta-\nu &  \text{ for } \delta-1 \leq \nu=\delta, \ \text{ see Theorem \ref{easycase}}.  
\end{cases}
\end{equation}

Under suitable conditions, this collection of inertia forms gives a system of generators of the moving curve ideal of $\phi$. 
The following proposition is an extension to \cite[Proposition 4.2]{HSV} and \cite[Proposition 4.4]{HSV} which deal with the cases $d=4$ and $d=5$ respectively.

\begin{proposition}[$\mu=2$] If $\deg(\phi)=1$ and $d=4$ then $p$, $q$ and the collection of inertia forms \eqref{listTFcas2} form a system of generators of the moving curve ideal of $\phi$. Moreover, the same result holds if 
$\deg(\phi)=1$, $d>4$ and $V(U_0,U_1,U_2)=\emptyset \subset \PP^2_\KK$.
\end{proposition}
\begin{proof} If $\deg(\phi)=1$ then Equation \eqref{Res=SRes} implies that $$\mathrm{depth}(\SRes_{0}(p,q),\ldots,\SRes_{\delta}(p,q))\geq 2.$$
Indeed, $\Res(p,q)$ is an implicit equation of the curve $\Cc$ and it is irreducible. Therefore, $\Res(p,q)$ is a generator of $H^0_\mm(B)_0$ by Theorem \ref{prop-cas3}. 

Now, since $V(U_0,U_1,U_2)=\emptyset$ we deduce  that  $\mathrm{depth}_{\KK[\ud{T}]}(U_0,U_1,U_2) \geq 3$ and hence, from Theorem \ref{prop-cas2}, that the collection of inertia forms $\Delta_{\alpha_1,\ldots,\alpha_{\delta-\nu}}(p,q)$, $|\alpha_i|=\delta-\nu$ is a system of generators of $H^0_\mm(B)_\nu$ for all $1\leq \nu\leq \delta-2$. Finally, Theorem \ref{easycase} shows that $\sylv_{\alpha}(p,q), |\alpha|=\delta-\nu$, is a system of generators of $H^0_\mm(B)_\nu$ for $\nu=\delta-1$ and $\nu=\delta$. 
\end{proof}

Two comments are in order here. First, the hypothesis $V(U_0,U_1,U_2)=\emptyset$, implicitly assumed in \cite[Proposition 4.9]{HSV}, is not superfluous since otherwise there exist some counterexamples. Also, we  mention that  this latter condition corresponds to the geometric property that there is no singular point on the curve $\Cc$ of multiplicity $d-2$, the maximum possible value for a singular point on $\Cc$ in this case by \cite[Theorem 3]{SCG07}. 

Secondly, we showed that if $\deg(\phi)=1$, i.e.~$\phi$ is birational onto $\Cc$, then the greatest common divisor of the subresultants $\SRes_i(p,q)$, $i=0,\ldots,d-2$ is a non-zero constant. We can actually prove along the same line in \cite{BuDa04} that this is an equivalence. Moreover, in this case the inertia forms $X_1\SRes_i(p,q)-X_2\SRes_{i+1}(p,q)$, $i=0,\ldots,d-3$, yield rational maps from $\PP^2$ to $\PP^1$ that all induce the inverse of the parametrization $\phi$. 

\subsection{The case $\mu\geq 3$} From Section \ref{inform}, we have the following list of inertia forms of $p$ and $q$:
\begin{equation}\label{listTFcas3}
\begin{cases}
 D_\beta(p,q), |\beta|=\nu & \text{ for } 0 \leq \nu \leq \mu-2, \ \text{ Theorem \ref{prop-cas3}, } \\ 
 \Delta_{\alpha_1,\ldots,\alpha_{m}}(p,q), |\alpha_i|=\delta-\nu &  \text{ for } \mu-1\leq \nu \leq d-\mu-2, \ \text{ Theorem \ref{prop-cas2}}, \\
 \sylv_{\alpha}(p,q), |\alpha|=\delta-\nu &  \text{ for } d-\mu-1 \leq \nu=\delta, \ \text{ Theorem \ref{easycase}}.  
\end{cases}
\end{equation}

By Theorem \ref{easycase}, we know that the inertia forms $\sylv_{\alpha}(p,q), |\alpha|=\delta-\nu$, form a system of generators for $H^0_\mm(B)_\nu$ for $d-\mu-1 \leq \nu=\delta$. Also, by Theorem \ref{prop-cas3}, the collection of inertia forms $D_\beta(p,q)$, $|\beta|=\nu$, form a system of generators of $H^0_\mm(B)_\nu$ for all $\nu \leq \mu-2$ provided that 
\begin{equation}\label{depth-mu=3}
\mathrm{depth}_{A} \left( C_{d-2\mu} \oplus A \xrightarrow{( p \ q )} C_{d-\mu} \right)\geq \mu
\end{equation}
(in particular $\phi$ has to be birational onto $\Cc$). 
The latter inequality can only be satisfied if $\mu=3$ since $\mathrm{depth}_A(T_1,T_2,T_3)=3$. By Theorem \ref{prop-cas2}, the collection of inertia forms $\Delta_{\alpha_1,\ldots,\alpha_{m}}(p,q)$, with $m=\delta-\nu-\mu+2$ and  $|\alpha_i|=\delta-\nu$, is a system of generators for $H^0_\mm(B)_\nu$ for $\mu-1 \leq \nu\leq d-\mu-2$ if the inequality $\mathrm{depth}_A(U_0,\ldots,U_\mu)\geq \mu+1\geq 4$ holds. But such an inequality never holds in $A$. 

\subsection{Inertia forms of degree $1$.} We finally gather the results concerning the inertia forms of degree 1 that we will need in the next section.

If $\mu=1$ we have seen that $H^0_\mm(B)_1$ is isomorphic to $A$ and hence generated by a unique determinant, or equivalently by an iterated Sylvester form. 

If $\mu\geq 2$ the inertia forms of degree 1 that we have described are always built from subresultants. More precisely, we proved that for all $i=0,\ldots,\delta-1=d-3$, the polynomials 
$$X_1\SRes_i(p,q)-X_2\SRes_{i+1}(p,q)$$
are inertia forms of degree 1. Moreover, they generate $H^0_\mm(B)_1$
\begin{itemize}
\item if $\mu=2$ and $d=4$ (see Equations \eqref{nu=1:d=2}),
\item if $\mu=2$, $d\geq 5$ and $V(U_0,U_1,U_2)=\emptyset \subset \PP^2$,
\item  if $\mu\geq 3$ and 
\begin{equation*}
\mathrm{depth}_{A}\Det_{d-4}
\underbrace{\left(\begin{array}{ccc}
U_0 &  & 0 \\
 & \ddots &  \\
\vdots & &  U_0 \\
  & & \\ 
U_{\mu} & & \vdots \\
 & \ddots & \\
0 & & U_{\mu}
\end{array}\right.
}_{d-\mu-2}
\underbrace{\left.\begin{array}{ccc}
V_0 &  & 0 \\
 & \ddots &  \\
\vdots & &  V_0 \\
  & & \\ 
V_{d-\mu} & & \vdots \\
 & \ddots & \\
0 & & V_{d-\mu}
\end{array}\right)
}_{\mu-2}
 \geq 3,
\end{equation*}
(notice that the above matrix has $d-2$ rows).
\end{itemize}

As we will mention in the next section, there are examples that show that such conditions are necessary.

\section{Adjoint pencils}\label{adj}

As in Section \ref{eqRees}, suppose given the parametrization \eqref{phi}
\begin{eqnarray*}
 \PP_\KK^1 & \xrightarrow{\phi} & \PP_\KK^2 \\
(X_1:X_2) & \mapsto & (g_1:g_2:g_3)(X_1,X_2)
\end{eqnarray*}
and assume that the greatest common divisor of $g_1,g_2,g_3$ over $\KK[X_1,X_2]$ is a non-zero constant in $\KK$. Moreover, we will also assume hereafter that $\phi$ is birational\footnote{As a consequence of Lur\"oth's Theorem, any rational curve can be properly re-parametrized, so this condition is not restrictive.} onto its image, that is to say the curve $\Cc$, and that 
$\KK$ is an algebraically closed field. 
We recall that 
$p(\ud{X},\ud{T}),$ $q(\ud{X},\ud{T})$ denote a $\mu$-basis of the parametrization $\phi$, where $p$, resp.~$q$, has degree $\mu$, resp.~$d-\mu$, in the variables $X_1,X_2$ and $1\leq \mu\leq d-\mu$. 

Since $\Cc$ is a rational plane curve, it is well-known that the genus of $\Cc$ is zero, that is to say that 
\begin{equation}\label{genus}
\frac{(d-1)(d-2)}{2}=\sum_{\pp \in \mathrm{Sing}(\Cc)}\frac{m_\pp(m_\pp-1)}{2}
\end{equation}
where the sum is over all the singular points, proper as well as infinitely near, of $\Cc$ and  $m_\pp$ denotes the multiplicity of $\Cc$ at $\pp$. Notice that to distinguish the infinitely near singularities of $\Cc$, we call a \emph{proper} singularity of $\Cc$ a usual singular point of $\Cc$ in the $(T_1:T_2:T_3)$-projective plane.

\begin{definition}\label{def:adjoint} An algebraic curve $\Dc$ is said to be \emph{adjoint} to $\Cc$ if $\Dc$ is going with virtual multiplicity $m_\pp-1$ through all the singular points, proper as well as infinitely near, of $\Cc$ of multiplicity $m_\pp$. 
\end{definition}
The notions of \emph{virtual multiplicity} and \emph{virtually going through} are quite subtle and essential to formulate a correct and useful inductive definition of adjoint curves. However, since we will not handle these notions in the sequel, we will not go further into the details and refer the reader to \cite[Sections 4.1 and 4.8]{Casas}. We just mention that if $\pp$ is a proper singular point of $\Cc$ of multiplicity $m_\pp$, then a curve $\Dc$ goes through $\pp$ with virtual multiplicity $m_\pp-1$ if it has multiplicity at least $m_\pp-1$ at $\pp$. Therefore, it is clear what is an adjoint to a curve having no  infinitely near singularity.

\medskip

The curve $\Cc$ being rational, it can be shown that curves adjoint to $\Cc$ of degree $\leq d-3$ do not exist, whereas curves adjoint to $\Cc$ of degree $\geq d-2$ are guaranteed to exist. Of course, the character of curves adjoint to $\Cc$ of degree $d-2$ and $d-1$ is particularly interesting.

An \emph{adjoint pencil} on $\Cc$ of degree $m$ is a one-parameter family of curves adjoint to $\Cc$ of degree $m$. It is hence of the form $X_1D_1(\ud{T})+X_2D_2(\ud{T})$ where $D_1,D_2$ are homogeneous polynomials in $\KK[\ud{T}]$ of degree $m$. Recently, David Cox noticed that moving curves of $\phi$ following $\Cc$ of degree $d-2$ (resp.~$d-1$) that are linear in $X_1,X_2$ sometimes give adjoint pencils on $\Cc$ (we refer the reader to  \cite[Conjecture 3.8 and Remark 3.9]{Cox07} for precise statements). 
Denote by $\Lc_{d-2}(\phi)$ (resp.~$\Lc_{d-1}(\phi)$) the finite $\KK$-submodule of the moving curve ideal of $\phi$ consisting of moving curves of degree $d-2$ (resp.~$d-1$) that are linear in $X_1,X_2$.
In what follows, using the results of Section \ref{eqRees} we determine explicit moving curves in $\Lc_{d-2}(\phi)$ and  $\Lc_{d-1}(\phi)$ that give adjoint pencils on $\Cc$. Point out that a similar study could be done for moving curves of degree $2,3,\mathrm{etc}$ in $X_1,X_2$ using the same approach. However, we will stick to the case of moving curves linear in $X_1,X_2$ because of its geometric content.

\medskip

According to the notation of Section \ref{eqRees}, the study of $\Lc_{d-2}(\phi)$ and $\Lc_{d-1}(\phi)$ relies on the study of the elements in $H^0_\mm(B)_1$, where $A=\KK[\ud{T}]$, $B=A[X_1,X_2]/(p,q)$ and $\mm=(X_1,X_2)$, that are homogeneous of degree $d-2$ or $d-1$ in the variables $T_1,T_2,T_3$. In the rest of the paper, we will always assume
that $d\geq 3$; this is not restrictive because a rational curve of degree $\leq 2$ has no singular point.  Finally, observe that since we are assuming that $\phi$ is birational onto $\Cc$, the inertia forms of degree 0 of $p,q$ with respect to the ideal $(X_1,X_2)$ are generated by $\Res(p,q) \in \KK[\ud{T}]$ which is an irreducible and homogeneous polynomial of degree $d$. Therefore, the inertia forms of degree 0 of $p,q$ w.r.t.~$X_1,X_2$ do not contribute to $\Lc_{d-2}(\phi)$ or $\Lc_{d-1}(\phi)$. 

We begin with the simple case $\mu=1$ before turning to the case $\mu \geq 2$ for which we will need to give another characterization of adjoint curves.

\subsection{The case $\mu=1$} The $\mu$-basis associated to the parametrization $\phi$ of $\Cc$ is of the form 
\begin{align*}
 p(\underline{X},\underline{T})&= \sum_{i=1}^3 p_i(\underline{X})T_i = U_0(\underline{T})X_1+U_1(\underline{T})X_2,\\
q(\underline{X},\underline{T})&= \sum_{i=1}^3 q_i(\underline{X})T_i = V_0(\underline{T})X_1^{d-1}+V_1(\underline{T})X_1^{d-2}X_2+\cdots+V_{d-1}(\underline{T})X_2^{d-1}.
\end{align*}

\begin{lemma}[$\mu=1, d\geq 3$]\label{lem:mu=1} The curve $\Cc$ has a unique (proper) singular point $\pp$ (of multiplicity $d-1$). Moreover, $U_0(\pp)=U_1(\pp)=0$.
\end{lemma}
\begin{proof}
Since $d\geq 3$, Equation \eqref{genus} implies that there exists at least one singular point on $\Cc$. By a suitable linear change of coordinates one may assume that this point is at the origin: $\pp=(0:0:1)$. 
Then, we claim that $p_3=0$ in \eqref{pq}. If this is true, clearly $U_0(\pp)=U_1(\pp)=0$ and we will have
$$C(T_1,T_2,T_3)=\Res(p,q)=\Res(p_1T_1+p_2T_2,q)$$
that shows that $\ord_\pp C(T_1,T_2,T_3)\geq d-1$, i.e.~$\pp$ is a singular point of multiplicity $ \geq d-1$. Then, it will follow by \eqref{genus} that $\pp$ has multiplicity exactly $d-1$ and that it is the unique singular point of $\Cc$.

To prove that $p_3=0$ we proceed by contradiction and assume that $p_3\neq 0$. Since $C(\pp)=0$, we deduce that $\Res(p_3,q_3)=0$ and hence that $p_3$ divides $q_3$. Therefore, by a change of $\mu$-basis if necessary, we can assume that $q_3=0$. But then, by inspecting the Sylvester matrix of $p$ and $q$ we have
$$ C(T_1,T_2,1)=\Res(p,q)=\Res(p_3,q_1T_1+q_2T_2)+R(T_1,T_2)$$
where $\ord_\pp R(T_1,T_2) \geq 2$. Since $\pp$ is a singular point, the term $\Res(p_3,q_1T_1+q_2T_2)$ must vanish, that is $p_3$ must divide $q=q_1T_1+q_2T_2$, a contradiction with the fact that the couple $(p,q)$ is a $\mu$-basis of $\phi$.
\end{proof}

In Section \ref{Mnu} we defined  the matrix 
$$
M_1(p,q)=\underbrace{\left(\begin{array}{ccc}
U_0 &  & 0 \\
 & \ddots &  \\
U_1 & &  U_0 \\
  & \ddots & \\ 
0 & & U_1 \\
\end{array}\right.
}_{d-3}
\underbrace{\left.\begin{array}{c}
\vdots \\
 \\
q_\beta(X) \\
\\
\vdots \\
\end{array}\right)}_{1}
$$
where $\sum_{|\beta|=d-3}q_\beta(\ud{X})Y^\beta=\morl_{d-3,1}(p,q)$, and we proved that 
$$\det(M_1(p,q)) \in \Lc_{d-1}(\phi).$$
The two following propositions recover \cite[Theorem 3.2]{HW} -- see also Proposition \ref{cas2-d=1}.

\begin{proposition}[$\mu=1, d=3$]  The element 
$p$ is a $\KK$-basis of $\Lc_{1}(\phi)$ and gives an adjoint pencil on  $\Cc$. Moreover, any element in $\Lc_{2}(\phi)$, which is $\KK$-generated by $T_1p, T_2p, T_3p$ and $\sylv_{(0,0)}(p,q)$, gives an adjoint pencil on $\Cc$.
\end{proposition}
\begin{proof}
From Section \ref{sec:u=1}, we know that  $\Lc_{1}(\phi) =\langle p \rangle_\KK$ and that 
$$\Lc_{2}(\phi)=\langle T_1p, T_2p, T_3p, \sylv_{(0,0)}(p,q) \rangle_\KK.$$
By Lemma \ref{lem:mu=1}, $\pp$ is the unique singular point of $\Cc$, it has multiplicity $2$ and $U_0(\pp)=U_1(\pp)=0$. Moreover, $\sylv_{(0,0)}(p,q) \in (U_0,U_1)$ by construction
and hence it also vanishes at $\pp$.
\end{proof}

\begin{proposition}[$\mu=1, d>3$]
The element $\det(M_1(p,q)) \in \Lc_{d-1}(\phi)$ gives an adjoint pencil on $\Cc$.
\end{proposition}
\begin{proof}  By construction $q_\beta \in (U_0,U_1)$. Therefore, from the definition of $M_1(p,q)$ we deduce that 
$$ \det(M_1(p,q)) \in (U_0,U_1)^{d-2}.$$
Now, if $\pp$ be is a singular point of $\Cc$, then by Lemma \ref{lem:mu=1} $\pp$ is unique with multiplicity $d-1$ and $U_0(\pp)=U_1(\pp)=0$. Therefore, $\det(M_1(p,q))$ vanishes at $\pp$ with multiplicity at least $d-2$. Notice that $\det(M_1(p,q))\neq 0$ for it is a generator of $H^0_\mm(B)_1$ which can not be zero since $\Res(p,q) \in H^0_\mm(B)_0$ is non-zero.
\end{proof}

In general, if $d>3$ and $\mu=1$ an element of $\Lc_{d-2}(\phi)\oplus_\KK \Lc_{d-1}(\phi)$ is not an adjoint pencil on $\Cc$, but one can always find one, namely $\det(M_1(p,q))$. Indeed, from Section \ref{sec:u=1} we have 
$$\Lc_{d-2}=\langle \left( \ud{T}^\alpha p\right)_{|\alpha|=d-3}  \rangle_\KK \ \textrm{ and } \ 
\Lc_{d-1}=\langle \left( \ud{T}^\alpha p\right)_{|\alpha|=d-2},\det(M_1(p,q)) \rangle_\KK.$$
So the element $p$ of the $\mu$-basis has degree $\mu=1$ and hence contribute to $\Lc_{d-2}(\phi)\oplus_\KK \Lc_{d-1}(\phi)$ without producing an adjoint pencil on $\Cc$ (see e.g.~\cite[Example 3.7]{Cox07}).

\subsection{Adjoint and polar curves} Instead of using directly Definition \ref{def:adjoint} to show that a certain curve $\Dc$ is adjoint to $\Cc$, we will use a property of adjoint curves that allows us to prove that $\Dc$ is adjoint to $\Cc$ by looking at the intersection of $\Cc$ and $\Dc$ at all the proper singularities. This approach has the advantage of avoiding the consideration of the infinitely near singularities of $\Cc$ through a desingularization process of $\Cc$. To state this property, we first need to fix some notation. 

Given two plane curves  $\Dc$ and $\Dc'$ that intersect in a finite set of points, we denote by $\mult_\pp(\Dc,\Dc')$ the intersection multiplicity of $\Dc$ and $\Dc'$ at the point $\pp$, and by $\mult_\pp(\Dc)$ the multiplicity of $\Dc$ at $\pp$. Recall that
$$\mult_\pp(\Dc)=\min_{\Lc \textrm{ line through } \pp} \mult_\pp(\Dc,\Lc)$$
where the minimum is taken over all the lines $\Lc$ passing through the point $\pp$, and also that $$\mult_\pp(\Dc)=\ord_\pp D(T_1,T_2,T_3)$$
where $D(T_1,T_2,T_3)$ is an implicit equation of $\Dc$. 

Given $\gamma=(\alpha(x),\beta(x))$ a branch of $\Dc$ centered at $\pp$, the intersection multiplicity of $\gamma$ and $\Dc'$ at $\pp$ is defined as
$$\mult_\pp(\gamma,\Dc')=\ord_\pp \, D'(\alpha(x),\beta(x)).$$
Then, the multiplicity of $\gamma$ at $\pp$ is
$$\mult_\pp(\gamma)=\min_{\Lc \textrm{ line through } \pp} \mult_\pp(\gamma,\Lc)$$
and $\mult_\pp(\Dc)$ is equal to the sum of the multiplicities of the branches of $\Dc$.

\begin{definition} Suppose given a curve $\Dc \subset \PP^2$ with equation $D(T_1,T_2,T_3)=0$ and a point $\qq=(\qq_1:\qq_2:\qq_3) \in \PP^2$. The \emph{polar curve} of $\Cc$ w.r.t.~$\qq$ is the curve defined by the equation
$$\qq_1\frac{\partial D}{\partial T_1}+\qq_2\frac{\partial D}{\partial T_2}+\qq_3\frac{\partial D}{\partial T_3}=0.$$
\end{definition}

\begin{proposition}[{\cite[Theorem 6.3.1]{Casas}}]\label{Dedekind}
 Let $\Cc \subset \PP^2$ be a curve, $\qq \in \PP^2$ be a point not lying on $\Cc$ and $\Pc_\qq$ be the polar curve of $\Cc$ w.r.t.~$\qq$. A curve $\Dc \subset \PP^2$ is adjoint to $\Cc$ if and only if
$$\mult_\pp(\gamma,\Dc) \geq \mult_\pp(\gamma,\Pc_\qq)-\mult_\pp(\gamma,\Lc_\qq)+1$$
for all proper singular points $\pp$ of $\Cc$ and all branches $\gamma$ of $\Cc$ centered at $\pp$, where $\Lc_\qq$ denotes the line joining the points $\pp$ and $\qq$.
\end{proposition}

It should be noticed that the quantity $\mult_\pp(\gamma,\Pc_\qq)-\mult_\pp(\gamma,\Lc_\qq)$ is independent of the choice of the point $\qq \notin \Cc$. We also recall  that a birational parametrization of a plane algebraic curve gives naturally parametrizations for all the branch curves. In particular, the number of irreducible branches at a singular point $\pp$ is the number of its distinct pre-images under the parametrization (see for instance \cite[Proposition 3.7.8]{Casas}).

\subsection{The case $\mu \geq 2$}  An implicit equation of the curve $\Cc$ of degree $d$ is given by $\Res(p,q) \in \KK[\underline{T}]$, where $(p,q)$ is a $\mu$-basis. The next result shows that the first-order subresultants $\SRes_i(p,q)$, $i=0,\ldots,d-2$, define curves of degree $d-2$ that are adjoint to $\Cc$ (notice that since $2\leq \mu \leq d-\mu$, we must have $d\geq 4$). To prove this, we will need the following lemma that can be found in \cite[Lemma 5.1]{BuMo07}; we include the proof for the convenience of the reader.

\begin{lemma}\label{dREs=JSres} Suppose we are given two homogeneous polynomials 
\begin{align*}
g_1(X_1,X_2) &= a_{d_1}X_1^{d_1}+a_{d_1-1}X_1^{d_1-1}X_2+\cdots+a_{1}X_1X_2^{d_1-1}+a_0X_2^{d_1}, \\ \notag
g_2(X_1,X_2) &= b_{d_2}X_1^{d_2}+b_{d_2-1}X_1^{d_2-1}X_2+\cdots+b_{1}X_1X_2^{d_2-1}+b_0X_2^{d_2},
\end{align*}
of degree $d_1,d_2 \geq 2$, respectively, and with coefficients $a_i$'s and $b_j$'s in $R[T]$ where $R$ is a commutative ring. 
Then, we have the following equality in $R[T,x]$:
\begin{multline*}
\frac{\partial \Res(g_{1},g_{2})}{\partial T} = \\
 (-1)^{d_{1}+d_{2}}\left|\begin{array}{cc}
\frac{\partial g_1}{\partial T}(x,1) & \frac{\partial g_2}{\partial T}(x,1) \\
\frac{\partial g_1}{\partial X_1}(x,1) & \frac{\partial g_2}{\partial X_1}(x,1) \\ 
\end{array}\right|
{\SRes}_0(g_{1},g_{2}) \text{ \rm modulo } (g_{1}(x,1), g_{2}(x,1)). 
\end{multline*}
\end{lemma}
\begin{proof} 
Consider the polynomials
\begin{align*} 
g_1(X_1+xX_2,X_2) &= a^x_{d_1}X_1^{d_1}+a^x_{d_1-1}X_1^{d_1-1}X_2+\cdots+a^x_{1}X_1X_2^{d_1-1}+a^x_0X_2^{d_2}, \\ \notag
g_2(X_1+xX_2,X_2) &= b^x_{d_2}X_1^{d_2}+b^x_{d_2-1}X_1^{d_2-1}X_2+\cdots+b^x_{1}X_1X_2^{d_2-1}+b^x_0X_2^{d_2},
\end{align*}
where the $a_i^x$'s and the $b_j^x$'s are polynomials in $R[T,x]$. By the base change for\-mu\-la for subresultants \cite{Hong97}, we have the equalities in $R[T,x]$:
\begin{align*}
\Res(g_1(X_1+xX_2,X_2),g_2(X_1+xX_2,X_2)) &= \Res(g_1(X_1,X_2),g_2(X_1,X_2)), \\
\SRes_{0}(g_1(X_1+xX_2,X_2),g_2(X_1+xX_2,X_2)) &= \SRes_{0}(g_1(X_1,X_2),g_2(X_1,X_2)).
\end{align*}
Therefore, expanding the determinant
\begin{multline*}
\Res(g_1(X_1+xX_2,X_2),g_2(X_1+xX_2,X_2))= \\
\left|
\begin{array}{ccccccc}
  a^x_{d_1} & 0 & \cdots & 0 & b^x_{d_2} & 0 & 0 \\
  a^x_{d_1-1} & a^x_{d_1} &      & \vdots & b^x_{d_2-1} & \ddots &    0   \\
  \vdots &  & \ddots &  0  & \vdots &  & b^x_{d_2}      \\
  a^x_0 &  &  & a^x_{d_1} & b^x_{1} &  & b^x_{d_2-1}  \\ 
  0 & a^x_0  &  & a^x_{d_1-1} & b^x_{0} &  &   \vdots  \\
 \vdots&   & \ddots  & \vdots & 0 & \ddots & b^x_{1}    \\
   0 &  \cdots & 0  & a^x_0 & 0 & 0  &   b^x_0  \\
\end{array}\right|
\end{multline*}
with respect to its two last rows, we get
$$ 
\Res(g_{1},g_{2}) = (-1)^{d_1+d_2}
\left|\begin{array}{ccc}
a^x_0 & b^x_0   \\ 
a^x_1 & b^x_1  \\ 
\end{array}\right|\SRes_{0}(g_{1},g_{2}) 
+ (a^x_{0})^{2} \Delta_{1} + a^x_{0} b^x_{0} \Delta_{2} + (b^x_{0})^{2} \Delta_{3},
$$
where $\Delta_{i}\ (i=1,2,3)$ are polynomials in the  $a^x_{i}$'s and
$b^x_{j}$'s. Taking the derivative with respect to the variable $T$, we deduce that
\begin{equation}\label{eq:lemme}
\frac{\partial \Res(g_{1},g_{2})}{\partial T} = (-1)^{d_1+d_2}
\left|\begin{array}{ccc}
\partial_2a^x_0 & \partial_2b^x_0   \\ 
a^x_1 & b^x_1  \\ 
\end{array}\right|\SRes_{0}(g_{1},g_{2}) \text{ modulo }
(a^x_0,b^x_0)\end{equation}
in $R[T,x]$. But it is easy to check that
$$a^x_0=g_1(x,1), \ b^x_0=g_2(x,1), \ a^x_1=\frac{\partial g_1}{\partial X_1}(x,1) \  \ b^x_1= \frac{\partial g_2}{\partial X_1}(x,1)$$
and therefore to deduce that \eqref{eq:lemme} is the claimed equality.
\end{proof}

\begin{theorem}[$\mu \geq 2, d\geq 4$]\label{SRes:adjoint} For all $i=0,\ldots,d-2$, the equation $\SRes_i(p,q)=0$ defines a plane curve which is adjoint to $\Cc$.
\end{theorem}
\begin{proof} Recall that the curve $\Cc$ of degree $d$ is parametrized by the generically injective rational map 
$$\PP^1 \xrightarrow{\phi} \PP^2 : (X_1:X_2) \mapsto (g_1:g_2:g_3)(X_1,X_2)$$
and that $(p,q)$ stands for a $\mu$-basis of $\phi$. Hereafter, we will denote by $\Dc$ the curve defined by $\SRes_0(p,q)=0$. It is well-defined by properties of the first principal subresultant and the $\mu$-basis. We first prove that the curve $\Dc$ defined by $\SRes_0(p,q)=0$ is adjoint to $\Cc$ by using the characterization of adjoint curves given in Proposition \ref{Dedekind}. 

So let $\pp$ be a proper singular point of $\Cc$ and $\gamma$ be an irreducible branch of $\Cc$ centered at $\pp$.
By a linear change of coordinates in $\PP^2$, we can assume that $\pp$ is the origin $(0:0:1)$ and that the point $\qq=(0:1:0)$ does not belong to $\Cc$. Also,
by a linear change of coordinates in $\PP^1$, we can assume that $\phi(0:1)=\pp$; turning to the affine parameter $(x:1) \in \PP^1$ with $x \in \KK$, we have $\phi(0)=\pp$ and we can assume that 
$$\ord_\pp \phi_1(x) \leq \ord_\pp \phi_2(x), \ \ \ord_\pp \phi_3(x)\neq 0$$
where $(\phi_1(x):\phi_2(x):\phi_3(x))$ is a local parametrization of $\gamma$.

The polar curve $\Pc_\qq$ of $\Cc$ with respect to $\qq$ is the curve of equation $$\frac{\partial \Res(p,q)}{\partial T_2}(T_1,T_2,T_3)=0$$
and therefore 
$$\mult_\pp(\Pc_\qq,\gamma)= \ord_\pp \left(\frac{\partial \Res(p,q)}{\partial T_2}(\phi_1(x),\phi_2(x),\phi_3(x))\right).$$
Similarly, the line $\Lc_\qq$ joining the points $\pp$ and $\qq$ is the line of equation $T_1=0$ and hence
$$\mult_\pp(\Lc_\qq,\gamma)=\ord_\pp \phi_1(x).$$
Also, we have
$$\mult_\pp(\Dc_\qq,\gamma)= \ord_\pp \left(\SRes_0(p,q)(\phi_1(x),\phi_2(x),\phi_3(x))\right).$$
Now, set
$$J(x)=\left|
\begin{array}{cc}
p_2(x) & q_2(x) \\
\sum_{i=1}^3 p_i'(x)\phi_i(x) & \sum_{i=1}^3 q_i'(x)\phi_i(x) \\ 
\end{array}
\right| \in \KK[x]
$$
where the notation $f'(x)$ stands for the derivative of the polynomial $f(x)$ with respect to the variable $x$. 
Since $(p,q)$ is a $\mu$-basis of $\phi$, we have $\sum_{i=1}^3p_i(x)\phi_i(x)=0$ and $\sum_{i=1}^3q_i(x)\phi_i(x)=0$. Therefore, Lemma \ref{dREs=JSres} shows that
$$\mult_\pp(\Dc,\gamma) = \mult_\pp(\Pc_\qq,\gamma) - \ord_\pp J(x).$$
Since $\phi_1(x)=\left| \begin{array}{cc} p_2(x) & q_2(x) \\ p_3(x) & q_3 (x) \end{array}\right|$, we have 
$$\phi_1'(x)=\left| \begin{array}{cc} p_2'(x) & q_2'(x) \\ p_3(x) & q_3 (x) \end{array}\right| +
\left| \begin{array}{cc} p_2(x) & q_2(x) \\ p_3'(x) & q_3'(x) \end{array}\right|.$$
Furthermore, the equality $\sum_{i=1}^3 p_i(x)\phi_i(x)=0$ shows that we have $\ord_\pp p_3(x) \geq \ord_\pp \phi_1(x)$ since $\ord_\pp \phi_2(x)\geq \ord_\pp \phi_1(x)$. Similarly, $\ord_\pp q_3(x) \geq \ord_\pp \phi_1(x)$ and we deduce that
$$\ord_\pp \left| \begin{array}{cc} p_2(x) & q_2(x) \\ p_3'(x) & q_3'(x) \end{array}\right| = \ord_\pp \phi_1(x) -1.$$
But
$$J(x)=\left|\begin{array}{cc}
p_2(x) & q_2(x) \\
p_1'(x) & q_1'(x) \\ 
\end{array}
\right| \phi_1(x)
+
\left|\begin{array}{cc}
p_2(x) & q_2(x) \\
p_2'(x) & q_2'(x) \\ 
\end{array}
\right| \phi_2(x)
+
\left|\begin{array}{cc}
p_2(x) & q_2(x) \\
p_3'(x) & q_3'(x) \\ 
\end{array}
\right| \phi_3(x)$$
and since $\ord_\pp(\phi_3)=0$ we get $\ord_\pp J(x)=\ord_\pp \phi_1(x) -1$. Finally, we deduce that
$$\mult_\pp(\Dc,\gamma) = \mult_\pp(\Pc_\qq,\gamma)-\mult_\pp(\Lc_\qq,\gamma) + 1$$
that proves that $\Dc$ is adjoint to $\Cc$ by Proposition \ref{Dedekind}.

To finish the proof, we need to consider the equations $\SRes_i(p,q)=0$ for all $i=1,\ldots,d-2$. As observed in Section \ref{eqRees}, under the hypothesis $\mu\geq 2$ the polynomials 
$X_1\SRes_i(p,q)-X_2\SRes_{i+1}(p,q)$ are inertia forms of $(p,q)$ with respect to the ideal $(X_1,X_2)$. It follows that for all $i=0,\ldots,d-3$, there exist an integer $N_i$ such that
$$X_1^{N_i}(X_1\SRes_i(p,q)-X_2\SRes_{i+1}(p,q)) \in (p,q) \subset \KK[\underline{X}][\underline{T}]$$
$$X_2^{N_i}(X_1\SRes_i(p,q)-X_2\SRes_{i+1}(p,q)) \in (p,q) \subset \KK[\underline{X}][\underline{T}]$$
From the properties of the curve $\Dc$ of equation $\SRes_0(p,q)=0$ we just proved, we deduce incrementally that for all $i=1,\ldots,d-2$, the equation $\SRes_{i}(p,q)=0$ defines an algebraic curve, say $\Dc_i$, such that
$$ \mult_\pp (\Dc_i,\gamma) \geq \mult_\pp(\Dc,\gamma) $$
(notice that it is actually almost always an equality) at each proper singular point $\pp$ of the curve $\Cc$. 
\end{proof}


We are now ready to state results on the relation between 
adjoint pencils on $\Cc$ and moving curves following $\phi$ of degree 1 in $X_1,X_2$ and degree $d-2$ (resp.~$d-1$) in $T_1,T_2,T_3$.

\begin{corollary}[$\mu=2, d=4$]\label{u2d4}
Any non-zero element in $\Lc_{2}(\phi)$ or $\Lc_{3}(\phi)$ gives an adjoint pencil on $\Cc$.
\end{corollary}
\begin{proof}
  By Section \ref{eqRees}, if $d=4$ we know that the two moving curves
$$\sylv_{0,1}(p,q)=
\left| 
\begin{array}{cc}
U_0 & U_2 \\
V_0 & V_2
 \end{array}
\right| X_1 +
\left| 
\begin{array}{cc}
U_1 & U_2 \\
V_1 & V_2
 \end{array}
\right| X_2=\SRes_2(p,q)X_2-\SRes_1(p,q)X_1$$
$$\sylv_{1,0}(p,q)=
\left| 
\begin{array}{cc}
U_0 & U_1 \\
V_0 & V_1
 \end{array}
\right| X_1+
\left| 
\begin{array}{cc}
U_0 & U_2 \\
V_0 & V_2
 \end{array}
\right| X_2=\SRes_0(p,q)X_1-\SRes_1(p,q)X_2
$$
form a system of generators of $\Lc_{2}(\phi)$ over $\KK$ and that $$\Lc_{3}(\phi)=\left\langle \left(T_i\,\sylv_{(0,1)}(p,q)\right)_{i=1,2,3},\left(T_i\,\sylv_{(1,0)}(p,q)\right)_{i=1,2,3} \right\rangle_\KK.$$ 
Therefore, the claimed result follows from that fact that  $\SRes_i(p,q)=0$ defines a curve adjoint to $\Cc$ for all $i=0,1,2$.
\end{proof}

The case $\mu=2, d>4$ is more intricate. Following Section \ref{Mnu}, we have to consider the $(d-2)\times (d-3)$-matrix
$$
M_1(p,q)=\underbrace{\left(\begin{array}{ccc}
U_0 &  & 0 \\
 & \ddots &  \\
U_1 & &  U_0 \\
  & \ddots & \\ 
U_{2} & & U_1 \\
 & \ddots & \\
 & & U_{2}
\end{array}\right.
}_{d-4}
\underbrace{\left.\begin{array}{c}
 \\
\vdots \\
 \\
q_\beta(X) \\
\\
\vdots \\
\\
\end{array}\right).}_{1}
$$
Its first $(d-4)$ columns are the coefficients of the polynomials 
$$X_1^{d-5}p, X_1^{d-6}X_2p, \ldots, X_1X_2^{d-6}p, X_2^{d-5}p$$
in the monomial basis $X_1^{d-3},X_1^{d-2}X_2,\ldots,X_2^{d-3}$ and its last column contains   
the coefficients of the polynomials $\sum_{|\beta|=d-3} q_{\beta}(X)Y^\beta$ in the monomial basis $Y_1^{d-3}$, $Y_1^{d-2}Y_2, \ldots, Y_2^{d-3}$. 
We proved in Theorem \ref{prop-cas2} that the maximal minors $\Delta_{i}$, $i\in \{1,\ldots,d-2\}$, of $M_1(p,q)$ obtained by removing the $i^{\text{th}}$ row are elements in $\Lc_{2}(\phi)$. Moreover, we proved in Lemma \ref{SRes:mu=2} that 
 $$\Delta_{i}=X_1\SRes_{i}(p,q)-X_{2}\SRes_{i+1}(p,q)$$
for all $i=1,\ldots,d-2$. We deduce the

\begin{corollary}[$\mu=2, d>4$]
All the determinants $\Delta_i$, $i\in \{1,\ldots,d-2\}$, give adjoint pencils $\Cc$. 
Moreover, if  $V(U_0,U_1,U_2)=\emptyset \subset \PP^2_\KK$ then
any non-zero element in $\Lc_{2}(\phi)$ or $\Lc_{3}(\phi)$ gives an adjoint pencil on $\Cc$.
\end{corollary}

Notice that the hypothesis on the depth of the ideal $(U_0,U_1,U_2)$ can not be avoided since XiaHong Jia found an example of an element in $\Lc_{3}(\phi)$, with $\mu=2$, $d=6$, which is not an adjoint pencil on $\Cc$. For this example, $p$ does not depend on the variable $T_3$ and hence $\mathrm{depth}_{\KK[\ud{T}]}(U_0,U_1,U_2)\leq 2$. Also, as mentioned earlier, notice that this condition corresponds to the absence of a (proper) singularity of multiplicity  $d-2$, the maximum possible value.

\medskip

Finally, let us consider the case $\mu \geq 3$. Since $1\leq \mu \leq d-\mu$, we must have $d\geq 6$. In Section \ref{sec:cas3} we have identified the collection of degree 1 inertia forms $\left(D_{i}(X_1,X_2)\right)_{i=0,\ldots,d-3}$ that satisfy to the equalities (see Corollary \ref{nu=1:cas3})
$$D_i=X_2\SRes_{d-2-i}(p,q)-X_1\SRes_{d-3-i}(p,q).$$

\begin{corollary}[$\mu\geq 3$] The polynomials $D_i(X_1,X_2)$, $i=0,\ldots,d-3$, give adjoint pencils on $\Cc$. Moreover, if 
\begin{equation}\label{lasteq}
V\left(\Det_{d-4}\left( C_{d-\mu-3}\oplus C_{\mu-3} \xrightarrow{(p \ q)} C_{d-3} \right)\right)=\emptyset \subset \PP^2_\KK
\end{equation}
then any element in $\Lc_{d-2}(\phi)$ or $\Lc_{d-1}(\phi)$ give an adjoint pencil on $\Cc$.
\end{corollary}

We point out that XiaHong Jia found an example of an element in $\Lc_{d-2}(\phi)\oplus_\KK \Lc_{d-1}(\phi)$, with $\mu=3$, $d=7$, which is not an adjoint pencil of $\Cc$. For this example the equality \eqref{lasteq} does not hold (more precisely, $p$ does not depend on the variable $T_3$ for this example). 

\subsection{Abhyankhar's Taylor resultant}

In its book \cite[Lecture 19]{Ab}, Abhyankar defines the \emph{Taylor resultant} of two polynomials $f(t), g(t) \in \KK[t]$ as the resultant that eliminates the variable $t$ from the two polynomials 
\begin{eqnarray}\label{TRf}
	\frac{f(t)-f(s)}{t-s}&=&f'(s)+\frac{f''(s)}{2!}t+\frac{f'''(s)}{3!}t^2+\cdots \\
	\label{TRg}
	\frac{g(t)-g(s)}{t-s}&=&g'(s)+\frac{g''(s)}{2!}t+\frac{g'''(s)}{3!}t^2+\cdots
\end{eqnarray}
(notice that the above equalities hold if $\KK$ has characteristic zero). We will denote it by $\Delta(t) \in \KK[t]$. As stated in the theorem page 153, this Taylor resultant is a generator of the conductor of  $k[f(t),g(t)]$ in $k[t]$. In particular, assuming $(f(t),g(t))$ to be a parametrization of the irreducible plane curve $\Cc$, it yields the singularities of $\Cc$ counted properly, that is to say that
$$\Delta(t)= \gamma \prod_{j=1}^l (t-\gamma_j)^{\epsilon_j}$$
where $0\neq \gamma \in \KK$, $\gamma_1,\ldots,\gamma_l$ are distinct elements in $\KK$, $\epsilon_1,\ldots,\epsilon_l$ are positive integers and it holds that
\begin{itemize}
	\item[\rm (P1)] $P=(\alpha,\beta) \in \Cc$ is a (proper) singular point on $\Cc$ if and only if $(\alpha,\beta)=(f(\gamma_j),g(\gamma_j))$ for some $j\in \{1,\ldots,l\}$
	\item[\rm (P2)] If $P=(\alpha,\beta) \in \Cc$ is a (proper) singular point on $\Cc$ then
	$$ {\sum}^{(\alpha,\beta)} \epsilon_j = {\sum}^{P} \nu_i(\nu_i-1)$$
	where $ {\sum}^{(\alpha,\beta)}$ is the sum over those $j$ for which $(\alpha,\beta)=(p(\gamma_j),q(\gamma_j))$, and ${\sum}^{P}$ is the sum over those $i$ for which the point $T_i$ either equals $P$ or is infinitely near to $P$ and has multiplicity $\nu_i$.
	\item[\rm (P3)] $\deg(\Delta(t))=(\deg(\Cc)-1)(\deg(\Cc)-2)$ if all the singular points of $\Cc$, proper as well as infinitely near, are at finite distance.
\end{itemize} 
Notice that Abhyankar used the right side of the equations \eqref{TRf} and \eqref{TRg} and hence required characteristic zero for the ground field $\KK$. In \cite{EY}, this point was overcome by considering the resultant of the left side of these equations, and the Taylor resultant was renamed the \emph{D-resultant}. Then, in \cite{Elkahoui} the D-resultant is expressed in terms of a certain subresultant.

Abhyankar's Taylor resultant, or $D$-resultant, only treats rational curves that admit a polynomial parametrization. Therefore, one can ask for a generalization of this resultant for any rational plane curve, that is to say for the case where $f(t)$ and $g(t)$ are not polynomials but rational functions in $\KK(t)$. In \cite{GRY} a partial answer to this question is proposed: If $f(t)=f_n(t)/f_d(t)$ and $g(t)=g_n(t)/g_d(t)$, the authors define the Taylor resultant as the resultant that eliminates $t$ from the two polynomials 
\begin{equation}\label{eq}
 \frac{f_n(t)f_d(s)-f_d(t)f_n(s)}{t-s}, \ \frac{g_n(t)g_d(s)-g_d(t)g_n(s)}{t-s}.
\end{equation}
However, this generalization does not represent exactly the singular points of the rational curve parametrized by $(f(t),g(t))$. Indeed, Equations \eqref{eq} introduce a symmetry that mixes the four rational curves parametrized by  $x=f(t)^{\pm 1}, y=g(t)^{\pm 1}$ (see \cite[Theorem 3.1]{GRY} for more details). 

\medskip

Using the notation \eqref{phi}, we claim that the polynomial 
\begin{equation}\label{DDres}
 \Delta(t)=\SRes_0(p,q)(g_1(t,1),g_2(t,1),g_3(t,1)) \in \KK[t]	
\end{equation}
provides an appropriate extension of the Taylor resultant to the rational case, having exactly the same properties (P1), (P2), (P3). Indeed, in the proof of Theorem \ref{SRes:adjoint} it is proved that the inequalities of Proposition \ref{Dedekind} are all equalities when the curve $\Dc$ is the curve defined by the polynomial $\SRes_0(p,q)(T_1,T_2,T_3)$. This guarantees that its intersection with the curve $\Cc$ and  is the adjoint divisor of $\Cc$ and therefore that the properties (P1), (P2), (P3) hold.

It should be noticed that the polynomial \eqref{DDres} already appeared in the paper 
\cite{CS01} where the authors showed that its roots are in correspondence with the proper singularities of $\Cc$. However, they did not prove the adjunction property of this polynomial.

\section{Acknowledgments}
I would like to thank David Cox, Carlos D'Andrea, Ron Goldman for helpful conversations on moving curve ideals and XiaHong Jia for all the interesting examples she computed. I am especially grateful to Eduardo Casas-Alvero for fruitful discussions on adjoint curves, but also for its very useful comments on a preliminary version of this manuscript.


\begin{thebibliography}{vdEY97}

\bibitem[Abh90]{Ab}
Shreeram~S. Abhyankar.
\newblock {\em Algebraic geometry for scientists and engineers}, volume~35 of
  {\em Mathematical Surveys and Monographs}.
\newblock American Mathematical Society, Providence, RI, 1990.

\bibitem[ASV81]{ASV81}
Jos\'e~F. Andrade, Aron Simis, and Wolmer~V. Vasconcelos.
\newblock On the grade of some ideals.
\newblock {\em Manuscripta Math.}, 34(2-3):241--254, 1981.

\bibitem[BD04]{BuDa04}
Laurent Bus{\'e} and Carlos D'Andrea.
\newblock Inversion of parameterized hypersurfaces by means of subresultants.
\newblock In {\em ISSAC 2004}, pages 65--71. ACM, New York, 2004.

\bibitem[BJ03]{BuJo03}
Laurent Bus{\'e} and Jean-Pierre Jouanolou.
\newblock On the closed image of a rational map and the implicitization
  problem.
\newblock {\em J. Algebra}, 265(1):312--357, 2003.

\bibitem[BM09]{BuMo07}
Laurent Bus{\'e} and Bernard Mourrain.
\newblock Explicit factors of some iterated resultants and discriminants.
\newblock {\em Mathematics of Computations}, 78:345--386, 2009.

\bibitem[CA00]{Casas}
Eduardo Casas-Alvero.
\newblock {\em Singularities of plane curves}, volume 276 of {\em London
  Mathematical Society Lecture Note Series}.
\newblock Cambridge University Press, Cambridge, 2000.

\bibitem[CHW08]{HW}
David~A. Cox, J.~William Hoffman, and Haohao Wang.
\newblock Syzygies and the {R}ees algebra.
\newblock {\em J. Pure Appl. Algebra}, 212(7):1787--1796, 2008.

\bibitem[Cox08]{Cox07}
David~A. Cox.
\newblock The moving curve ideal and the {R}ees algebra.
\newblock {\em Theoret. Comput. Sci.}, 392(1-3):23--36, 2008.

\bibitem[CS01]{CS01}
Eng-Wee Chionh and Thomas~W. Sederberg.
\newblock On the minors of the implicitization {B}\'ezout matrix for a rational
  plane curve.
\newblock {\em Comput. Aided Geom. Design}, 18(1):21--36, 2001.

\bibitem[EK05]{Elkahoui}
M'Hammed El~Kahoui.
\newblock {$D$}-resultant and subresultants.
\newblock {\em Proc. Amer. Math. Soc.}, 133(8):2193--2199 (electronic), 2005.

\bibitem[GRY02]{GRY}
Jaime Gutierrez, Rosario Rubio, and Jie-Tai Yu.
\newblock {$D$}-resultant for rational functions.
\newblock {\em Proc. Amer. Math. Soc.}, 130(8):2237--2246 (electronic), 2002.

\bibitem[Hon97]{Hong97}
Hoon Hong.
\newblock Subresultants under composition.
\newblock {\em Journal of Symbolic Computation}, 23(4):355--365, 1997.

\bibitem[HSV08]{HSV}
Jooyoun Hong, Aron Simis, and Wolmer~V. Vasconcelos.
\newblock On the homology of two-dimensional elimination.
\newblock {\em J. Symbolic Comput.}, 43(4):275--292, 2008.

\bibitem[Jou96]{Jou96}
Jean-Pierre Jouanolou.
\newblock R\'esultant anisotrope, compl\'ements et applications.
\newblock {\em Electron. J. Combin.}, 3(2):Research Paper 2, approx.\ 91 pp.\
  (electronic), 1996.
\newblock The Foata Festschrift.

\bibitem[Jou97]{Jou97}
Jean-Pierre Jouanolou.
\newblock Formes d'inertie et r\'esultant: un formulaire.
\newblock {\em Adv. Math.}, 126(2):119--250, 1997.

\bibitem[Jou07]{Jou07}
Jean-Pierre Jouanolou.
\newblock An explicit duality for quasi-homogeneous ideals.
\newblock Preprint arXiv:math.AC/0607626, to appear in {\it J. of Symb. Comp.},
  2007.

\bibitem[SCG07]{SCG07}
Ning Song, Falai Chen, and Ron Goldman.
\newblock Axial moving lines and singularities of rational planar curves.
\newblock {\em Comp. Aided Geom. Design}, 24(4):200--209, 2007.

\bibitem[Vas94]{Vas94}
Wolmer~V. Vasconcelos.
\newblock {\em Arithmetic of blowup algebras}, volume 195 of {\em London
  Mathematical Society Lecture Note Series}.
\newblock Cambridge University Press, Cambridge, 1994.

\bibitem[vdEY97]{EY}
Arno van~den Essen and Jie-Tai Yu.
\newblock The {$D$}-resultant, singularities and the degree of unfaithfulness.
\newblock {\em Proc. Amer. Math. Soc.}, 125(3):689--695, 1997.

\end{thebibliography}

\end{document}